\documentclass[journal,twoside,web]{ieeecolor}
\pdfoutput=1

\usepackage{generic}
\usepackage{cite}
\usepackage{amsmath,amssymb,amsfonts}
\usepackage{algorithmic}
\usepackage{graphicx}
\usepackage{algorithm,algorithmic}
\usepackage[hidelinks=true,colorlinks=true,citecolor=blue,linkcolor=blue,urlcolor=blue]{hyperref}
\usepackage{textcomp}

\usepackage{bm}
\usepackage{ascmac}
\usepackage{here}
\usepackage{color}
\usepackage{mathrsfs}
\usepackage[subrefformat=parens]{subcaption}

\usepackage{mathtools}

\newtheorem{problem}{Problem}
\newtheorem{proposition}{Proposition}

\newtheorem{lemma}{Lemma}
\newtheorem{theorem}{Theorem}
\newtheorem{corollary}{Corollary}
\newtheorem{rmk}{Remark}

\newtheorem{assumption}{Assumption}

\newcommand{\calA}{{\mathcal A}}

\newcommand{\calH}{{\mathcal H}}

\newcommand{\calN}{{\mathcal N}}

\newcommand{\calR}{{\mathcal R}}    
\newcommand{\calS}{{\mathcal S}}    
    
\newcommand{\calU}{{\mathcal U}}

\newcommand{\scrR}{{\mathscr R}}

\newcommand{\bbE}{{\mathbb E}}

\newcommand{\bbH}{{\mathbb H}}

\newcommand{\bbP}{{\mathbb P}}    
    
\newcommand{\bbR}{{\mathbb R}}

\newcommand{\bbZ}{{\mathbb Z}}

\newcommand{\diag}{\mathop{\rm diag}\nolimits}

\newcommand{\rmd}{{\rm d}}

\newcommand{\rmr}{{\rm r}}

\newcommand{\what}[1]{{\widehat{#1}}}
 
\newcommand{\wtilde}[1]{{\widetilde{#1}}}

\newcommand{\bbra}[1]{\ensuremath{[\![#1]\!]} }  % [[  ]]

\newcommand{\kl}[2]{D_{\rm KL} \left({#1}\|{#2}\right)}
\newcommand{\ft}{N}

\newcommand{\bsigma}{{\bar{\Sigma}}}
\newcommand{\bara}{{\bar{A}}}
\newcommand{\barq}{{\bar{Q}}}

\newcommand{\cro}{{S}}
\newcommand{\esu}{{D}}

\newcommand{\cyan}[1]{{\color{black}{#1}}} %????
\newcommand{\red}[1]{{\color{black}{#1}}} %????
\newcommand{\rev}[1]{{\color{black}{#1}}} %????
\newcommand{\modi}[1]{{\color{black}{#1}}} %????
\newcommand{\revv}[1]{{\color{black}{#1}}} %????
\newcommand{\modii}[1]{{\color{black}{#1}}} %????
\newcommand{\fin}[1]{{\color{black}{#1}}} %????
\newcommand{\gram}[1]{{\color{black}{#1}}} %????

\newcommand{\dlt}[1]{} %????

\newcommand{\schr}{{Schr\"{o}dinger}}

\def\BibTeX{{\rm B\kern-.05em{\sc i\kern-.025em b}\kern-.08em
    T\kern-.1667em\lower.7ex\hbox{E}\kern-.125emX}}
\begin{document}
\title{Maximum Entropy Density Control of Discrete-Time Linear Systems with \\Quadratic Cost}
\author{Kaito Ito, \IEEEmembership{Member, IEEE}, and Kenji Kashima, \IEEEmembership{Senior Member, IEEE}
\thanks{This work was supported in part by JSPS KAKENHI Grant Numbers JP23K19117, JP21H04875, and by JST, ACT-X Grant Number JPMJAX2102.}
\thanks{K. Ito is with the Department of Information Physics and Computing, The University of Tokyo, Tokyo, Japan (e-mail: kaito@g.ecc.u-tokyo.ac.jp). }
\thanks{K. Kashima is with the Graduate School of Informatics, Kyoto University, Kyoto, Japan (e-mail: kk@i.kyoto-u.ac.jp).}
% \thanks{Third C. Author is with 
% the Electrical Engineering Department, University of Colorado, Boulder, CO 
% 80309 USA, on leave from the National Research Institute for Metals, 
% Tsukuba, Japan (e-mail: author@nrim.go.jp).}
}

\maketitle

\begin{abstract}
    This paper addresses the problem of steering the distribution of the state of a discrete-time linear system to a given target distribution while minimizing an entropy-regularized cost functional.
	This problem is called a maximum entropy density control problem.
	Specifically, the running cost is given by quadratic forms of the state and the control input, and the initial and \fin{target} distributions are Gaussian. We first reveal that our problem boils down to solving two Riccati difference equations coupled through their boundary values. Based on them, we give the closed-form expression of the unique optimal policy.
    Next, we show that the \fin{optimal} density control of a \modi{backward} system can be obtained simultaneously with the forward-time optimal policy.
    \fin{The backward solution gives another expression of the forward solution.}
	Finally, by considering the limit where the entropy regularization vanishes, we derive the \fin{unregularized density control in closed form.}
\end{abstract}


\section{Introduction}\label{sec:intro}
\IEEEPARstart{T}{his} paper is concerned with steering the distribution of the state of a linear dynamical system to a desired distribution while minimizing a cost functional. The distribution represents probabilistic uncertainty, and hence this kind of problem is referred to as uncertainty control~\cite{Chen2021review}, and has potential applications for example to uncertainty-aware guidance and planning~\cite{Ridderhof2018,Okamoto2019,Knaup2023,Shirai2023}.
Recently, steering a state distribution to a target distribution has gained considerable attention in generative modeling~\cite{Song2021,De2021}.
Especially when the distribution has a density function, uncertainty control is also called density control, where the density can be interpreted as the distribution of a large number of identical agents~\cite{Krishnan2018,Chen2018mean,Brockett2012}.
\red{Uncertainty control considering only the first and second moments of the state distribution is referred to as covariance control (steering)~\cite{Grigoriadis1997,Collins1987,Bakolas2016,Goldshtein2017,Liu2022}.}

\red{Several works have been devoted to the density control and covariance control problems for both continuous-time and discrete-time systems.
However, the cases in which the optimal policy is explicitly given are quite limited.}
In \cite{Chen2016part1}, minimum energy density control of continuous-time linear stochastic systems between Gaussian distributions was considered, and the explicit form of the optimal policy was derived. 
The work~\cite{Chen2018} also obtained the closed-form expression of the optimal policy in the case where the performance index has a quadratic state cost in addition to a control cost.

\rev{
    \fin{In} discrete-time cases, \cite{Goldshtein2017} and \cite{Balci2023} derived explicit solutions for the covariance control of deterministic systems with quadratic cost in the form of the initial state feedback policy.
\fin{However, until recently}, the explicit forms of the optimal {\em current} state feedback policies for density or covariance control problems \fin{had} not been found, unlike in the continuous-time cases.}
\fin{Recently}, \cite{Ito2022maximum} considered an {\em entropy-regularized} version of the minimum energy density control of a discrete-time linear system and obtained the optimal current state feedback policy in closed form.

Entropy-regularized optimal control is called maximum entropy (MaxEnt) control~\cite{Levine2018,Kim2023}. The entropy regularization encourages higher entropy of the optimal policy. MaxEnt control has attracted much attention especially in reinforcement learning (RL). This is because the entropy regularization brings many advantages for RL such as random exploration considering control performance~\cite{Haarnoja2017}, robustness against disturbances~\cite{Eysenbach2021}, and equivalence between the optimal control problem and an inference problem~\cite{Levine2018,Ito2024risk}.
Despite the benefits of the regularization, the resulting high-entropy policy increases the state uncertainty, which severely limits its applicability for example to safety-critical systems.
MaxEnt density control enables us to tame the state uncertainty while using high-entropy policies.

{\it Contributions:}
In this paper, we extend the result in \cite{Ito2022maximum} to the case where a quadratic state cost and \rev{a cross term between the state and the control input are} also present.
The analysis in \cite{Ito2022maximum} is based on coupled Lyapunov equations.
When considering a state cost \modi{and a cross term}, this approach can no longer be employed. Instead, we show that a key ingredient for the analysis of our problem is coupled Riccati equations, which are nonlinear and thus require a very different approach for their analysis from the Lyapunov equations.
The main contributions of this paper are as follows:
\begin{enumerate}
	\item[1)] We tackle the MaxEnt density control problem for discrete-time linear deterministic systems with Gaussian initial and target distributions. \rev{First, we reveal that this problem can be reduced to solving coupled Riccati equations. Although this kind of reduction is known for the continuous-time problem whose cost does not have the cross term~\cite{Chen2018}, it is still valid for discrete-time systems with the cross term.} \rev{By analyzing the Riccati equations}, we show the existence and uniqueness of the optimal policy, and derive its explicit form. 
    \item[2)] \red{We consider the MaxEnt density control of \rev{a backward system, which becomes the time reversal of the original system in the absence of the cross term of the cost.} Then, we reveal that the \modi{backward} density control problem can be solved simultaneously with the forward-time MaxEnt density control problem.} \rev{This result is a generalization of the fact that the classical \schr \ bridge (SB) problem~\cite{Leonard2014}, which is equivalent to a continuous-time density control problem, simultaneously solves its time-reversed density control problem~\cite{Chen2021liaisons}.}
	\item[3)] As a limiting case of the MaxEnt density control, we investigate the density control without entropy regularization. By considering the limit of vanishing regularization, we obtain the optimal \modii{current state feedback policy of} the unregularized density control problem.
\end{enumerate}

{\it Related work:}
We now introduce related work on density (covariance) control with quadratic cost.
In \cite{Liu2022}, covariance control \fin{of} a discrete-time linear stochastic system was investigated, and the existence and uniqueness of the optimal solution \fin{were} established under the reachability of the system. It was also shown that the solution can be computed by solving a semi-definite programming.
In \cite{De2021discrete}, the authors considered the optimal transport problem~\cite{Villani2003} whose transport cost is given by the value function for the linear quadratic (LQ) optimal control.
Then, they constructed a suboptimal \fin{density control} policy of a discrete-time linear system with quadratic cost using the optimal transport map. Under some condition, this policy is optimal~\cite{Terpin2023}.
However, the closed-form expression of the transport map is not derived even for Gaussian end-point distributions.

In \cite{Bakolas2016}, the author applied a convex relaxation technique to covariance control \fin{of} a discrete-time linear system with constraints on the state and the control input.
The work \cite{Halder2016} considered density control of a linear system with Wasserstein terminal cost instead of a final density constraint, and obtained the optimality condition for a linear controller gain.
Other scenarios for covariance control have been studied, e.g., chance constraints~\cite{Okamoto2018}, nonlinear systems, and non-quadratic cost~\cite{Yi2020,Ridderhof2019}.

\cyan{In \cite{Chen2018}, the zero-noise limit for density control of a continuous-time linear system driven by a Wiener process was considered. As a result, the closed-form solution to the density control problem for the deterministic system without the Wiener process was obtained.}
The zero-noise limit for a stochastic single-integrator system with general end-point distributions and running cost was studied in \cite{Mikami2004,Mikami2008}.

\revv{
    The MaxEnt LQ density control problem without state costs is equivalent to a discrete-time SB problem~\cite{Ito2022maximum}. For continuous-time cases and for discrete time and space cases, SB problems reduce to solving the so-called \schr~system~\cite{Chen2021liaisons}. The work \cite{Chen2015anisotropic} studied the continuous-time LQ density control as a generalized SB and revealed that the associated \schr~system can be rewritten as coupled Riccati differential equations. Despite the extensive studies on the continuous-time SB, there has been little work on discrete-time and continuous-space SB problems~\cite{Beghi1996}.
    In this work, we directly derive coupled Riccati difference equations for the MaxEnt LQ density control rather than dealing with a \schr~system, which is not known for our problem.
}

\textit{Organization:}
The remainder of this paper is organized as follows. In Section~\ref{sec:formulation}, we give the problem formulation and preliminary analysis. 
In Section~\ref{sec:solution}, we derive the optimal solution to our problem.
In Section~\ref{sec:reverse}, we study the \rev{backward density control problem associated with the forward MaxEnt density control.}
In Section~\ref{sec:zero-noise}, we investigate the unregularized density control problem as a limiting case of the MaxEnt density control problem.
In Section~\ref{sec:example}, we give a numerical example to illustrate the obtained results.
Some concluding remarks are given in Section~\ref{sec:conclusion}.

\textit{Notation:}
Let $ \bbR $ denote the set of real numbers and $ \bbZ_{>0} $ denote the set of positive integers. The set of integers $ \{k,k+1,\ldots,l\} \ (k< l) $ is denoted by $ \bbra{k,l} $.
Denote by $ \calS^n $ the set of all real symmetric $ n\times n $ matrices.
For matrices $ A, B \in \calS^{n} $, we write $ A \succ B $ (resp. $A \prec B$) if $ A - B $ is positive (resp. negative) definite. Similarly, $ A \succeq B $ means that $ A - B $ is positive semidefinite.
For $ A \succeq 0 $, $ A^{1/2} $ denotes the unique positive semidefinite square root.
The identity matrix is denoted by $ I $, and its dimension depends on the context.
The transpose of the inverse of an invertible matrix $ A $ is denoted by $ A^{-\top} $.
\cyan{Denote the determinant of a square matrix $ A $ by $ \det (A) $.}
The Euclidean norm is denoted by $ \|\cdot \| $.
For $ x\in \bbR^n $ and \revv{$  A\succ 0 $} of size $ n $, denote $\|x\|_A := (x^\top A x)^{1/2}$.
Let $(\Omega, \mathscr{F}, \bbP)$ be a complete probability space and $\bbE$ be the expectation with respect to $ \bbP $.
\cyan{The Kullback--Leibler divergence between probability densities $ p_x, p_y $ is denoted by $ \kl{p_x}{p_y} $ when it is defined.}
For an $ \bbR^n $-valued random vector $ w $, $ w\sim \calN(\mu,\Sigma) $ means that $ w $ has a multivariate Gaussian distribution with mean $ \mu \in \bbR^n $ and covariance matrix $ \Sigma $. When $ \Sigma \succ 0 $, the density function of $ w\sim \calN (\mu,\Sigma) $ is denoted by $ \calN (\cdot | \mu, \Sigma) $.
We use the same symbol for a random variable and its realization.

\section{Problem Formulation and Preliminary Analysis}\label{sec:formulation}
In this paper, we tackle the following problem.
\begin{problem}[MaxEnt density control problem]\label{prob:density_control}
	Find a policy $ \pi = \{\pi_k\}_{k=0}^{\ft-1} $ that solves
	\begin{align}
            &\underset{\pi }{\rm minimize}  && \bbE \Biggl[ \sum_{k=0}^{\ft-1} \biggl( \frac{1}{2} \begin{bmatrix}
                x_k \\ u_k
            \end{bmatrix}^\top \rev{\begin{bmatrix}
                Q_k & \cro_k \\
                \cro_k^\top & R_k          
            \end{bmatrix}} \begin{bmatrix}
                x_k \\ u_k
            \end{bmatrix} \nonumber\\
         & &&\qquad\qquad - \varepsilon\calH(\pi_k (\cdot | h_k))   \biggr)  \Biggr] \label{eq:cost} \\
			&{\rm subject~to} && x_{k+1} = A_k x_k + B_k u_k, \label{eq:system}\\
				& &&u_k \sim \pi_k (\cdot | h_k) \ {\rm given} \ h_k, \label{eq:policy}\\
				& &&h_k := \{ x_0,u_0,\ldots, x_{k-1},u_{k-1} ,x_k\}, \label{eq:history}\\
				& &&\fin{k = 0,\ldots,\ft-1}, \nonumber\\
				& &&x_0 \sim \calN(0,\bar{\Sigma}_0), \  x_N \sim \calN(0,\bar{\Sigma}_\ft) , \label{eq:cov_constraint}
	\end{align}
	where $ \bar{\Sigma}_0 $, $\bar{\Sigma}_\ft \succ 0$, $R_k \succ 0 $,
    \rev{$
        Q_k - \cro_k R_k^{-1} \cro_k^\top \succeq 0 
    $},
    $ \{x_k\} $ is an $ \bbR^n $-valued state process, $ \{u_k\} $ is an $ \bbR^m $-valued control process, $ N\in \bbZ_{>0} $ is a finite horizon, and $ \varepsilon > 0 $ is a regularization parameter. A stochastic policy $ \pi_k (\cdot| h_k) $ denotes the conditional density of $ u_k $ given the history
	$ 
		h_k \in \bbH_k := \bbR^n \times \bbR^m \times \cdots \times \bbR^n \times \bbR^m \times \bbR^n ,
	$ 
	and $ \calH(\pi_k (\cdot|h_k)) := -\int_{\bbR^m} \pi_k(u|h_k) \log \pi_k (u|h_k) \rmd u $ denotes the entropy of $ \pi_k(\cdot | h_k) $.
	\hfill $ \diamondsuit $
\end{problem}

We will consider general mean constraints
\begin{equation}\label{eq:general_mean}
	x_0 \sim \calN(\bar{\mu}_0, \bsigma_0), \ x_\ft \sim \calN(\bar{\mu}_\ft, \bsigma_\ft)
\end{equation}
in Section~\ref{sec:solution}.
Note that, without loss of generality, we may assume $ R_k \equiv I, \ \varepsilon = 1 $. In fact, the weight parameter $ \varepsilon $ can be absorbed into $ Q_k, R_k, \cro_k $ as $ Q_{\varepsilon,k} := Q_k/\varepsilon $, $ R_{\varepsilon,k} := R_k/\varepsilon $, $ \cro_{\varepsilon,k} := \cro_k/\varepsilon $.
Moreover, to see the reason for $ R_k $, let $ \wtilde{u}_k := R_{\varepsilon,k}^{1/2} u_k $ and $ \wtilde{\pi}_k (\cdot | h_k ) $ be the conditional density of $ \wtilde{u}_k $ given $ h_k $. Then, it holds that
\begin{align*}
	\calH(\wtilde{\pi}_k (\cdot | h_k)) &= \calH (\pi_k (\cdot | h_k)) + \log  (\det (R_{\varepsilon,k}^{1/2})), \ \forall h_k  \in \bbH_k . 
\end{align*}
See e.g., \cite[Section 8.6]{Cover2006}. In summary, Problem~\ref{prob:density_control} can be rewritten as
\begin{align}
	&\underset{\tilde{\pi} = \{\tilde{\pi}_k \} }{\rm minimize}  && \bbE \Biggl[ \sum_{k=0}^{\ft-1} \biggl( \frac{1}{2} \begin{bmatrix}
        x_k \\ \wtilde{u}_k
    \end{bmatrix}^\top
    \begin{bmatrix}
        Q_{\varepsilon,k} & \cro_{R,\varepsilon,k} \\
        \cro_{R,\varepsilon,k}^\top & I
    \end{bmatrix}
    \begin{bmatrix}
        x_k \\ \wtilde{u}_k
    \end{bmatrix}
         \nonumber\\
		 & &&\qquad\qquad - \calH(\wtilde{\pi}_k (\cdot | h_k))   \biggr)  \Biggr]  \label{prob:transform}\\
	&{\rm subject~to} && \eqref{eq:cov_constraint}, \ x_{k+1} = A_k x_k + \sqrt{\varepsilon} B_k R_k^{-1/2} \wtilde{u}_k , \label{eq:system_trans}\\
	& &&\wtilde{u}_k \sim \wtilde{\pi}_k (\cdot | h_k) \ {\rm given} \ h_k , \nonumber\\
	& && h_k = \{x_0, R_{\varepsilon,0}^{-1/2} \wtilde{u}_0, \ldots, \nonumber\\
	& && \hspace{1cm} x_{k-1}, R_{\varepsilon,k-1}^{-1/2} \wtilde{u}_{k-1}, x_k \}, \nonumber\\
				& &&\fin{k = 0,\ldots,\ft-1} ,\nonumber
\end{align}
where $ \cro_{R,\varepsilon,k} := \cro_{k} R_k^{-1/2} / \sqrt{\varepsilon} $.

\rev{
To investigate Problem~\ref{prob:density_control}, we first study a MaxEnt control problem with a terminal state cost instead of the final density constraint~\eqref{eq:cov_constraint}. Its optimal policy is shown to be the sum of linear state-feedback control and Gaussian noise (Proposition~\ref{prop:LQ}). Thanks to the linearity and the Gaussianity, we see that the optimal density control of Problem~\ref{prob:density_control} can be obtained by solving coupled Riccati and Lyapunov equations, which describe the optimality of a policy and the time evolution of the state covariance, respectively. Moreover, we show that the Lyapunov equation can be transformed into a Riccati equation (Proposition~\ref{prop:riccati}). Consequently, the MaxEnt density control problem boils down to solving the coupled Riccati equations (Proposition~\ref{prop:opt}).
}

\modi{Now, we introduce the auxiliary problem without the final density constraint as follows.}
\begin{problem}[MaxEnt LQ control problem]\label{prob:LQ_control}
	Find a policy $ \pi = \{\pi_k\}_{k=0}^{\ft-1} $ that solves
	\begin{align}
			&\underset{\pi }{\rm minimize}  && \bbE \Biggl[ \frac{1}{2} x_\ft^\top Fx_\ft + \sum_{k=0}^{\ft-1} \biggl( \frac{1}{2} \begin{bmatrix}
                x_k \\ u_k
            \end{bmatrix}^\top \begin{bmatrix}
                Q_k & \cro_k \\
                \cro_k^\top & R_k          
            \end{bmatrix} \begin{bmatrix}
                x_k \\ u_k
            \end{bmatrix}\nonumber\\
         & &&\qquad -  \varepsilon\calH(\pi_k (\cdot | h_k))   \biggr)  \Biggr] \label{eq:cost_lq} \\
			&{\rm subject~to} && \eqref{eq:system}\text{--}\eqref{eq:history}, \nonumber
	\end{align}
	where $ F\in \calS^n $, and the initial state $ x_0 $ has a finite second moment.
	\hfill $ \diamondsuit $
\end{problem}

Similar to \cite[Proposition~1]{Ito2022maximum}, \modi{which deals with $ Q_k \equiv 0 $, $ \cro_k \equiv 0 $}, Problem~\ref{prob:LQ_control} can be solved by the dynamic programming approach. In this paper, rather than it, we adopt a completion of squares argument for the proof.
Let
\begin{align}
    \rev{\bara_k := A_k - B_k R_k^{-1} \cro_k^\top , \ \barq_k := Q_k - \cro_k R_k^{-1} \cro_k^\top.} \label{eq:barA}
\end{align}
\begin{proposition}\label{prop:LQ}
   Assume that $ \Pi_k \in \calS^n $ satisfies $ R_k + B_k^\top \Pi_{k+1} B_k \succ 0 $ for any $ k\in \bbra{0,\ft-1} $ and is a solution to the following Riccati difference equation:
   \begin{align}
      \Pi_k &= \bara_k^\top \Pi_{k+1} \bara_k - \bara_k^\top \Pi_{k+1} B_k (R_k + B_k^\top \Pi_{k+1}B_k)^{-1} \nonumber\\
      &\quad\times  B_k^\top \Pi_{k+1} \bara_k + \barq_k, ~~ k\in \bbra{0,\ft-1}, \label{eq:riccati_general}\\
      \Pi_\ft &= F .
   \end{align}
   Then, the unique optimal policy of Problem~\ref{prob:LQ_control} is given by
   \begin{align}
      &\pi_k^* (u|h_k) \nonumber\\
      &= \calN\bigl( u  |  -(R_k + B_k^\top \Pi_{k+1} B_k)^{-1} (B_k^\top \Pi_{k+1} A_k + \rev{\cro_k^\top } )x_k, \nonumber\\
      &\qquad\qquad \varepsilon(R_k + B_k^\top \Pi_{k+1} B_k)^{-1}\bigr) \label{eq:opt_policy}
   \end{align}
	 for any $ k \in \bbra{0,\ft-1}, \ u\in \bbR^m, $ and $ h_k \in \bbH_k  $.
	 In addition, the minimum value of \eqref{eq:cost_lq} is given by
	 \begin{align}
		&\bbE\left[\frac{1}{2} x_0^\top \Pi_0 x_0\right] \nonumber\\
		& - \varepsilon\sum_{k=0}^{\ft-1} \log \sqrt{(2\pi)^m \det (\varepsilon (R_k + B_k^\top \Pi_{k+1} B_k )^{-1})} .\label{eq:lq_minimum}
	 \end{align}
\end{proposition}
\begin{proof}
	Define $ V_k := R_k + B_k^\top \Pi_{k+1} B_k $.
	By a completion of squares argument, the cost functional~\eqref{eq:cost_lq} is rewritten as
	\begin{align}
        &\sum_{k=0}^{\ft-1} \biggl( \bbE\biggl[\frac{1}{2} \|u_k + V_k^{-1} (B_k^\top \Pi_{k+1} A_k + \cro_k^\top ) x_k \|_{V_k}^2 \nonumber\\
        &\quad- \varepsilon\calH (\pi_k(\cdot | h_k))   \biggr] \biggr) + \bbE\left[\frac{1}{2} x_0^\top \Pi_0 x_0\right] ,\label{eq:completion_square}
	\end{align}
	where we used \eqref{eq:system},~\eqref{eq:riccati_general}, and the second term does not depend on $ \pi $.
	Since $ V_k \succ 0 $ by assumption, the term in the parentheses is written as
	\begin{align*}
        &\rev{\bbE\biggl[\bbE\biggl[\frac{1}{2} \|u_k + V_k^{-1} (B_k^\top \Pi_{k+1} A_k + \cro_k^\top ) x_k \|_{V_k}^2} \\
        &\quad+ \varepsilon \int_{\bbR^m} \pi_k(u|h_k) \log \pi_k(u|h_k) \rmd u \ \biggl| \ h_k  \biggr] \biggr]\\
        &= \varepsilon\bbE\biggl[ \int \pi_k(u|h_k) \biggl( \frac{\|u + V_k^{-1} (B_k^\top \Pi_{k+1} A_k + \cro_k^\top )x_k \|_{V_k}^2}{2\varepsilon} \\
		&\quad + \log \pi_k(u|h_k) \biggr) \rmd u \biggr] \\
		&= \cyan{\varepsilon} \bbE \bigl[ D_{\rm KL}\bigl(\pi_k(\cdot|h_k) \bigl\| \\
        &\qquad \calN(\cdot | -V_k^{-1} (B_k^\top \Pi_{k+1} A_k + \cro_k^\top )x_k, \varepsilon V_k^{-1} )\bigr) \bigr] \\
		&\quad -  \varepsilon\log \sqrt{(2\pi)^m {\rm det}( \varepsilon V_k^{-1})} .
	\end{align*}
	Therefore, \eqref{eq:completion_square} attains its minimum \eqref{eq:lq_minimum} if and only if the policy is given by \eqref{eq:opt_policy}.
\end{proof}

By the above result, the optimal policy $ \pi_k (\cdot | h_k) $ requires only the current state $ x_k $. In the following, we will abuse the notation by writing $ \pi_k (\cdot | x_k) $.
As $ \varepsilon \searrow 0 $, \eqref{eq:opt_policy} reduces to the well-known LQ optimal controller~\cite{Lewis2012}
\[
	u_k = -(R_k + B_k^\top \Pi_{k+1} B_k)^{-1} (B_k^\top \Pi_{k+1} A_k + \cro_k^\top ) x_k,
\]
and the minimum value~\eqref{eq:lq_minimum} becomes $ \bbE[x_0^\top \Pi_0 x_0 /2] $.

For simplicity, we will assume $ R_k \equiv I $, $ \varepsilon = 1 $ for the reminder of this paper except in Section~\ref{sec:zero-noise}.
The optimal state process $ \{x_k^*\} $ for Problem~\ref{prob:LQ_control} driven by the policy \eqref{eq:opt_policy} is given by
\begin{align*}
   x_{k+1}^* &= \calA_k(\Pi_{k+1}) x_k^* + B_k w_k^*, \\
    w_k^* &\sim \calN(0,(I + B_k^\top \Pi_{k+1} B_k)^{-1}) ,
\end{align*}
where $ \{w_k^*\} $ is an independent sequence, and
\begin{align*}
  &\rev{\calA_k(\Pi_{k+1})} := (I + B_k B_k^\top \Pi_{k+1})^{-1} \bara_k \\
  &= A_k - B_k (I + B_k^\top \Pi_{k+1} B_k)^{-1} (B_k^\top \Pi_{k+1} A_k + \cro_k^\top ) .
\end{align*}
Then, the covariance matrix $ \Sigma_k := \bbE[x_k^* (x_k^*)^\top] $ evolves as
\begin{align}
   \Sigma_{k+1} &= \calA_k(\Pi_{k+1}) \Sigma_k \calA_k(\Pi_{k+1})^\top \nonumber\\
   &\quad + B_k (I + B_k^\top \Pi_{k+1} B_k)^{-1} B_k^\top , \label{eq:sigma_evolution}
\end{align}
and if the initial distribution is a zero-mean Gaussian, then $ x_k^* \sim \calN(0,\Sigma_k) $.
Now, assume that the optimal state process $ \{x_k^*\} $ for Problem~\ref{prob:LQ_control} satisfies the density constraints~\eqref{eq:cov_constraint} of Problem~\ref{prob:density_control}. Then, the policy \eqref{eq:opt_policy} is the unique optimal solution to Problem~\ref{prob:density_control}. \cyan{Indeed, assume that there exists a policy that satisfies the density constraints~\eqref{eq:cov_constraint} and yields a lower cost \eqref{eq:cost} for Problem~\ref{prob:density_control} than the policy \eqref{eq:opt_policy}. Then, it also leads to a lower cost \eqref{eq:cost_lq} for Problem~\ref{prob:LQ_control} than \eqref{eq:opt_policy} because for any policy satisfying \eqref{eq:cov_constraint}, the terminal cost $ \bbE[x_\ft^\top Fx_\ft /2] $ takes the same value. This contradicts the optimality of \eqref{eq:opt_policy}. \rev{Additionally, assume that there exists a feasible policy of Problem~\ref{prob:density_control} that is different from the optimal policy \eqref{eq:opt_policy}, but attains the minimum value of \eqref{eq:cost}. Then, it also attains the minimum value of \eqref{eq:cost_lq} due to the final density constraint~\eqref{eq:cov_constraint}. This contradicts the uniqueness of the solution to Problem~\ref{prob:LQ_control}.}}
In summary, the unique optimal density control policy \fin{of} Problem~\ref{prob:density_control} can be obtained by finding a terminal value $ \Pi_\ft = F $ of \eqref{eq:riccati_general} such that for any $ k\in \bbra{0,\ft-1} $, $ I + B_k^\top \Pi_{k+1} B_k \succ 0 $ holds, and the solution $ \Sigma_k $ to \eqref{eq:sigma_evolution} satisfies $ \Sigma_0 = \bsigma_0 $, $\Sigma_\ft = \bsigma_\ft $.

Next, we \modi{make the change of variables} $ H_k := \Sigma_k^{-1} - \Pi_k $ as in \cite{Chen2016part1,Ito2022maximum}.
Here, the invertibility of $ \Sigma_k $ is ensured if \rev{$ \bara_k $} is invertible for any $ k\in \bbra{0,\ft-1} $ and $ \Sigma_0 \succ 0 $ because $ \calA_k(\Pi_{k+1}) $ is also invertible, \modi{and it is assumed that $ I + B_k^\top \Pi_{k+1} B_k \succ 0 $} for any $ k\in \bbra{0,\ft-1} $.
Hence, we assume the following condition.
\rev{\begin{assumption}\label{ass:invertibility}
    For any $ k\in \bbra{0,\ft-1} $, $ \bara_k $ in \eqref{eq:barA} is invertible.
    \hfill $ \diamondsuit $
\end{assumption}
}

\modi{By using} $ H_k $, the forward equation~\eqref{eq:sigma_evolution} is transformed into a backward Riccati equation.
The proof is given in Appendix~\ref{app:riccati}.
\begin{proposition}\label{prop:riccati}
   Let $ \varepsilon = 1 $, $R_k = I $ for any $ k\in \bbra{0,\ft-1} $. \modi{Suppose that Assumption~\ref{ass:invertibility} holds} and assume $ \Sigma_0 \succ 0 $.
   Let $ \{\Sigma_k\} $ and $ \{\Pi_k\} $ be solutions to \eqref{eq:sigma_evolution} and
   \begin{align}
    \Pi_k &= \bara_k^\top \Pi_{k+1} \bara_k - \bara_k^\top \Pi_{k+1} B_k (I + B_k^\top \Pi_{k+1}B_k)^{-1} \nonumber\\
  &\quad\times  B_k^\top \Pi_{k+1} \bara_k + \barq_k, ~~ k\in \bbra{0,\ft-1}. \label{eq:riccati}
 \end{align}
   Then for $ H_k := \Sigma_k^{-1} - \Pi_k $, it holds that
   \begin{align}\label{eq:H_riccati}
      H_k &= \bara_k^\top H_{k+1} \bara_k - \bara_k^\top H_{k+1} B_k (-I + B_k^\top H_{k+1} B_k)^{-1} \nonumber\\
      &\quad \times B_k^\top H_{k+1} \bara_k - \barq_k, ~~ k \in \bbra{0,\ft-1} .
   \end{align}
   \hfill $ \diamondsuit $
\end{proposition}

Now, the boundary conditions $ \Sigma_0 = \bsigma_0 $, $\Sigma_\ft = \bsigma_\ft $ are transformed into
\begin{equation}\label{eq:boundary}
	\bar{\Sigma}_0^{-1} =  \Pi_0 + H_0, \ \bar{\Sigma}_\ft^{-1} =  \Pi_\ft + H_\ft.
\end{equation}
We summarize the above discussion in the following.
\begin{proposition}\label{prop:opt}
	Let $ \varepsilon = 1$, $R_k = I $ for any $ k\in \bbra{0,\ft-1} $.
    \modi{Suppose that Assumption~\ref{ass:invertibility} holds} and assume that $ \{\Pi_k\} $ and $ \{H_k\} $ satisfy the Riccati equations \eqref{eq:riccati},~\eqref{eq:H_riccati}
	 with the boundary conditions~\eqref{eq:boundary}. Assume further that $ I + B_k^\top \Pi_{k+1} B_k \succ 0 $ for any $ k\in \bbra{0,\ft-1} $. Then, the policy \eqref{eq:opt_policy} is the unique optimal policy of Problem~\ref{prob:density_control}.
   \hfill $ \diamondsuit $
\end{proposition}

\rev{\begin{rmk}\label{rmk:invertibility}
    \fin{In this paper}, the transformation of the evolution equation \eqref{eq:sigma_evolution} of $ \Sigma_k $ into the backward Riccati equation \eqref{eq:H_riccati} is crucial. If $ \bara_k $ is singular, $ \Sigma_{k} $ is not uniquely determined for a given $ \Sigma_{k+1} $ in general.
    Therefore, for the backward analysis of $ \Sigma_k $, the invertibility of $ \bara_k $ is essential.
    Note that $ \bara_k $ is expected to be invertible in many applications.
	For example when $ \cro_k = 0 $, if the system \eqref{eq:system} is obtained by a zero-order hold discretization of a continuous-time system, $ \bara_k = A_k $ is always invertible.
		The invertibility of $ A_k $ is also assumed in the literature on covariance steering~\cite{Liu2022,Balci2023multi,Balci2022exact}.
        \hfill $ \diamondsuit $
\end{rmk}}

\section{Solution to MaxEnt Density Control Problem}\label{sec:solution}
In this section, we derive the solutions to the Riccati equations \eqref{eq:riccati},~\eqref{eq:H_riccati} coupled through the boundary conditions~\eqref{eq:boundary} to obtain the optimal policy of Problem~\ref{prob:density_control} \modi{based on Proposition~\ref{prop:opt}}.
\rev{First, we transform the Riccati equations into coupled linear equations using a symplectic matrix (Proposition~\ref{prop:linear}). Then, we derive the explicit solutions to the linear equations, from which the solutions to the coupled Riccati equations can be recovered (Proposition~\ref{thm:riccati_solution}). In addition, Proposition~\ref{prop:positive} verifies that one of the obtained solutions fulfills the condition $ I + B_k^\top \Pi_{k+1} B_k \succ 0 $ required in Proposition~\ref{prop:opt} for the construction of the optimal policy. As a result, we obtain the closed-form expression of the optimal policy of Problem~\ref{prob:density_control} (Theorem~\ref{thm:opt_policy}).}

\modi{Let} us introduce $ X_k, Y_k, \widehat{X}_k, \widehat{Y}_k \in \bbR^{n\times n} $ satisfying for any $ k \in \bbra{0,\ft-1} $,
\begin{align}
	\begin{bmatrix}
		X_{k+1} \\ Y_{k+1}
	\end{bmatrix}
	&=\underbrace{
	\begin{bmatrix}
		\bara_k + G_k \bara_k^{-\top} \barq_k & - G_k \bara_k^{-\top} \\
		- \bara_k^{-\top} \barq_k & \bara_k^{-\top}
	\end{bmatrix}}_{=: M_k}
	\begin{bmatrix}
		X_k \\ Y_k
	\end{bmatrix}, \label{eq:XY}\\
   \begin{bmatrix}
		\what{X}_{k+1} \\ \what{Y}_{k+1}
	\end{bmatrix}
	&=
	M_k
	\begin{bmatrix}
		\what{X}_k \\ \what{Y}_k
	\end{bmatrix} , \label{eq:XYhat}
\end{align}
where $ G_k := B_k B_k^\top $. 
\fin{Note that the solutions to \eqref{eq:XY},~\eqref{eq:XYhat} are uniquely determined by terminal conditions $ (X_\ft,Y_\ft) $, $(\what{X}_\ft, \what{Y}_\ft) $ because} \modi{the symplectic matrix} $ M_k $ has the inverse
\begin{equation}\label{eq:M_inv}
    M_{k}^{-1} =
    \begin{bmatrix}
        \bara_{k}^{-1} & \bara_{k}^{-1} G_{k} \\
        \barq_{k} \bara_{k}^{-1} & ~ \barq_{k} \bara_{k}^{-1} G_{k} + \bara_{k}^\top
    \end{bmatrix} .
\end{equation}
The relationship between the linear equations~\eqref{eq:XY},~\eqref{eq:XYhat} and the Riccati equations~\eqref{eq:riccati},~\eqref{eq:H_riccati} is given as follows. The proof is provided in Appendix~\ref{app:linear}.
\begin{proposition}\label{prop:linear}
    \modi{Suppose that Assumption~\ref{ass:invertibility} holds.}
    Then, the following hold:
    \begin{itemize}
    \item[i)] Assume that $ \{\Pi_k\}_{k=0}^\ft$ and $\{H_k\}_{k=0}^\ft $ satisfy \eqref{eq:riccati},~\eqref{eq:H_riccati}.
    Let $ \{X_k,Y_k\}_{k=0}^\ft $ and $ \{\what{X}_k, \what{Y}_k\}_{k=0}^\ft $ be the solutions to \eqref{eq:XY},~\eqref{eq:XYhat} with terminal conditions $ (X_\ft,Y_\ft) $, $ (\what{X}_\ft, \what{Y}_\ft) $ satisfying $ \Pi_\ft = Y_\ft X_\ft^{-1} $, $ H_\ft = - \what{X}_\ft^{-\top} \what{Y}_\ft^{\top} $.
    Then, for any $ k \in \bbra{0,\ft} $, $ X_k $ and $ \what{X}_k $ are invertible, and it holds \fin{that} $ \Pi_{k} = Y_k X_k^{-1} $, $ H_k = - \what{X}_k^{-\top} \what{Y}_k^{\top} $.
    \item[ii)]
    \cyan{Let $ \{X_k,Y_k\}_{k=0}^\ft $ and $ \{\what{X}_k, \what{Y}_k\}_{k=0}^\ft $ be solutions to \eqref{eq:XY},~\eqref{eq:XYhat}, and assume that for any $ k\in \bbra{0,\ft} $, $ X_k $ and $ \what{X}_k $ are invertible. Then, $ \{\Pi_k\}_{k=0}^\ft $ and $ \{H_k\}_{k=0}^\ft $ given by $ \Pi_{k} := Y_k X_k^{-1} $, $ H_k := - \what{X}_k^{-\top} \what{Y}_k^{\top} $ are solutions to \eqref{eq:riccati},~\eqref{eq:H_riccati}, respectively.}
    \hfill $ \diamondsuit $
    \end{itemize}
\end{proposition}

We can always use the above expressions of $ \Pi_k, H_k $ for example by setting $ (X_\ft,Y_\ft) = (I,\Pi_\ft) $, $ (\what{X}_\ft, \what{Y}_\ft) = (I, -H_\ft) $.
Now, the boundary conditions~\eqref{eq:boundary} are rewritten as
\begin{align}
	&\bar{\Sigma}_0^{-1} = Y_0 X_0^{-1} - \what{X}_0^{-\top} \what{Y}_0^\top, \label{eq:boundary2_1} \\
	&\bar{\Sigma}_\ft^{-1} = Y_\ft X_\ft^{-1} - \what{X}_\ft^{-\top} \what{Y}_\ft^\top. \label{eq:boundary2_2} 
\end{align}
We analyze the Riccati equations~\eqref{eq:riccati},~\eqref{eq:H_riccati} with the boundary conditions~\eqref{eq:boundary} via the linear equations~\eqref{eq:XY}, \eqref{eq:XYhat} coupled through \eqref{eq:boundary2_1}, \eqref{eq:boundary2_2}.

Before stating the result, we establish some notations. Define the state-transition matrix for $ \{M_k\} $:
\begin{equation}\label{eq:transition_M}
	\Phi_M (k,l) :=
	\begin{cases}
		M_{k-1} M_{k-2} \cdots M_l, & k > l ,\\
		I, & k=l, \\
		M_k^{-1} M_{k+1}^{-1} \cdots M_{l-1}^{-1}, & k < l .
	\end{cases}
\end{equation}
Similarly, the state-transition matrix for $ \{\bara_k\} $ is denoted by $ \Phi_\bara(k,l) $.
We split $ \Phi_M (k,l) $ into $ n\times n $ matrices:
\begin{equation}\label{eq:split}
	\Phi_M (k,l) =
	\begin{bmatrix}
		\Phi_{11} (k,l) & \Phi_{12}(k,l) \\
		\Phi_{21} (k,l) & \Phi_{22}(k,l)
	\end{bmatrix} .
\end{equation}
In particular, $\Phi_{ij} (\ft,0)$, $\Phi_{ij} (0,\ft) $ ($ i,j = 1,2 $) are denoted by $ \phi_{ij}, \varphi_{ij} $, respectively.
Let us introduce the reachability Gramian:
\begin{equation}\label{eq:reachability}
	\calR(k_1,k_0) := \sum_{k = k_0}^{k_1-1} \Phi_{\rev{\bara}}(k_1,k+1) G_k \Phi_\bara(k_1,k+1)^\top , \ k_1 > k_0.
\end{equation}
For the analysis of the coupled Riccati equations, we assume the following \modi{reachability} condition, \rev{whose intuition will be explained in Remark~\ref{rmk:reachability}.}
\begin{assumption}\label{ass:reachability}
	There exists $ k_{\rmr} \in \bbra{1,\ft} $ such that $ \calR(k,0) $ is invertible for any $ k\in \bbra{k_\rmr,\ft} $ and $ \calR(\ft,k) $ is invertible for any \gram{$ k\in \bbra{0,k_\rmr} $}.
	\hfill$ \diamondsuit $
\end{assumption}

Now, we give the solutions to the Riccati equations \eqref{eq:riccati}--\eqref{eq:boundary}.
The proof is shown in Appendix~\ref{app:solution_riccati}.
\begin{proposition}\label{thm:riccati_solution}
	Suppose that Assumptions~\ref{ass:invertibility},~\ref{ass:reachability} hold.
	Define
	\begin{align}
		&Y_{\ft,\pm} := - \varphi_{12}^{-1} \varphi_{11} \nonumber\\
      &\pm \bsigma_\ft^{-1/2} \left( \frac{1}{4}I + \bsigma_\ft^{1/2} \varphi_{12}^{-1} \bsigma_0 \varphi_{12}^{-\top} \bsigma_\ft^{1/2}  \right)^{1/2} \bsigma_\ft^{-1/2} + \frac{1}{2} \bsigma_\ft^{-1} , \label{eq:Y_terminal} \\
			&Y_{0,\mp} := -\phi_{12}^{-1} \phi_{11} \nonumber\\
			&\pm \bsigma_0^{-1/2} \biggl( \frac{1}{4}I + \bsigma_0^{1/2} \phi_{12}^{-1} \bsigma_\ft \phi_{12}^{-\top} \bsigma_0^{1/2}  \biggr)^{1/2} \bsigma_0^{-1/2} + \frac{1}{2} \bsigma_0^{-1} . \label{eq:Y_0_pm}
	\end{align}
	Then, the Riccati equations~\eqref{eq:riccati},~\eqref{eq:H_riccati} with the boundary conditions~\eqref{eq:boundary} have at least one and at most two solutions.
    Specifically, \eqref{eq:riccati}--\eqref{eq:boundary} have a solution $ \{\Pi_k, H_k\} = \{\Pi_{k,+}, H_{k,+}\} $ specified by the terminal values:
	\begin{align}
		&\Pi_{\ft} = Y_{\ft,+}, \label{eq:Pi_terminal}\\
		&H_{\ft} = \bsigma_\ft^{-1} - Y_{\ft,+} . \label{eq:H_terminal}
	\end{align}
	In addition, $ \Pi_{k,+} $ satisfies $ \Pi_{0,+} = Y_{0,+} $.
	\cyan{Assume further that the equations~\eqref{eq:riccati},~\eqref{eq:H_riccati} with the terminal conditions~$ \Pi_\ft = Y_{\ft,-} $, $H_\ft = \bsigma_\ft^{-1} - Y_{\ft,-} $ have solutions $ \{\Pi_k\} = \{\Pi_{k,-}\} $, $ \{H_k\} = \{H_{k,-}\} $ on the time interval $ \bbra{0,\ft} $, respectively. Then, $ \{\Pi_{k,-}, H_{k,-}\} $ satisfies the boundary conditions~\eqref{eq:boundary}.}
	\hfill $ \diamondsuit $
\end{proposition}

Additionally, the following result ensures that the solution $ \{\Pi_{k,+}, H_{k,+}\} $ satisfies the assumption of the positive definiteness in Proposition~\ref{prop:opt}.
The proof is given in Appendix~\ref{app:positive}.

\begin{proposition}\label{prop:positive}
	Suppose that Assumptions~\ref{ass:invertibility},~\ref{ass:reachability} hold.
    Then, the following hold:
	\begin{itemize}
	\item[(i)] For any $ k\in \bbra{0,\ft-1} $, it holds that
	\begin{align}
		I + B_k^\top \Pi_{k+1,+} B_k &\succ 0, \label{eq:BPiB_positive} \\
		I - B_k^\top H_{k+1,+} B_k &\succ 0. \label{eq:BHB_positive} 
	\end{align}
	\item[(ii)] \cyan{Assume that the Riccati equation~\eqref{eq:riccati} with the terminal condition~$ \Pi_\ft = Y_{\ft,-} $ has a solution $ \{\Pi_{k,-}\}_{k=0}^\ft $.} Then, there exists $s \in \bbra{0, \ft -1}$ such that $I + B_s^\top \Pi_{s+1,-} B_s$ is not positive definite.
	\hfill $ \diamondsuit $
	\end{itemize}
\end{proposition}

The above result says even if the coupled Riccati equations \eqref{eq:riccati}--\eqref{eq:boundary} have the solution $ \{\Pi_k,H_k\} = \{\Pi_{k,-},H_{k,-}\} $, it cannot be used for the construction of the optimal policy for Problem~\ref{prob:density_control}.
Therefore, $ \{\Pi_{k,+}, H_{k,+} \} $ is the unique solution to \eqref{eq:riccati}--\eqref{eq:boundary} that satisfies $ I + B_k^\top \Pi_{k+1} B_k \succ 0 $ for any $ k\in \bbra{0,\ft-1} $.

By Propositions~\ref{prop:opt},~\ref{thm:riccati_solution},~\ref{prop:positive}, we obtain the main result of this section.
\begin{theorem}\label{thm:opt_policy}
	Suppose that Assumptions~\ref{ass:invertibility},~\ref{ass:reachability} hold. Then, the unique optimal policy of Problem~\ref{prob:density_control} with $ R_k \equiv I$, $\varepsilon = 1 $ is given by
	\begin{align}
		&\pi_k^* (u|x_k) \nonumber\\
        &= \calN\bigl( u  |  -(I + B_k^\top \Pi_{k+1,+} B_k)^{-1} (B_k^\top \Pi_{k+1,+} A_k + \rev{\cro_k^\top}) x_k, \nonumber\\
		&\qquad\qquad (I + B_k^\top \Pi_{k+1,+} B_k)^{-1}\bigr) \label{eq:opt_policy_thm}
 \end{align}
 for any $ k\in \bbra{0,\ft-1} $, $ u\in \bbR^m$, and $x_k\in \bbR^n $.
 Here, $ \{\Pi_{k,+}\} $ is the solution to \eqref{eq:riccati} with the terminal value \eqref{eq:Pi_terminal}.
 \hfill $ \diamondsuit $
\end{theorem}

In summary, to obtain the optimal policy of the MaxEnt density control problem (Problem~\ref{prob:density_control}), we only need to compute $ \Pi_\ft = Y_{\ft,+} $ given by \eqref{eq:Y_terminal} and the resulting solution $ \{\Pi_{k,+}\}_{k=0}^{\ft-1} $ to the Riccati equation~\eqref{eq:riccati}. The terminal value $ Y_{\ft,+} $ is determined by the initial and target covariance matrices $ \bsigma_0, \bsigma_\ft $ and the parameters \fin{$ \{\bara_k\} = \{A_k - B_k \cro_k^\top \} $}, $\{G_k \} = \{B_k B_k^\top \}$, \fin{$\{\barq_k\} = \{Q_k - \cro_k \cro_k^\top \} $} through $ \varphi_{11} = \Phi_{11}(0,\ft)$, $\varphi_{12} = \Phi_{12} (0,\ft) $. Recall that $ \Phi_{ij}(0,\ft) $ is the \mbox{$ (i,j) $-th} block of $ \Phi_M (0,\ft) = M_0^{-1}M_1^{-1} \cdots M_{\ft-1}^{-1} $, and $ M_k^{-1} $ is given by \eqref{eq:M_inv}.
Then, the optimal control $ u_k $ following \eqref{eq:opt_policy_thm} can be obtained as the addition of the LQ optimal control and independent Gaussian noise:
\begin{align*}
    u_k &= -(I + B_k^\top \Pi_{k+1,+} B_k)^{-1} (B_k^\top \Pi_{k+1,+} A_k + \cro_k^\top) x_k + w_k^*, \\
    w_k^* &\sim \calN(0, (I + B_k^\top \Pi_{k+1,+} B_k)^{-1}) .
\end{align*}
The optimal policy for the general case where $ R_k \neq I $, $ \varepsilon \neq 1 $ will be given in Corollary~\ref{cor:general_opt} of Section~\ref{sec:zero-noise}.

\rev{
\begin{rmk}\label{rmk:reachability}
    To consider the intuition behind Assumption~\ref{ass:reachability}, we consider $ \cro_k \equiv 0 $.
    Then, without the assumption on $ \calR(\ft,k) $, Assumption~\ref{ass:reachability} just requires the system~\eqref{eq:system} to be reachable over the horizon $ \ft $.
    The assumption on $ \calR(\ft,k) $ is made to resolve difficulties in our analysis due to the singularity of $ \calR(k,0) $ for small $ k $. For example, the invertibility of $ \calR(k,0) $ can be utilized to prove that of $ X_k $ used for $ \Pi_k = Y_k X_k^{-1} $; see Appendix~\ref{app:solution_riccati} together with Lemma~\ref{lem:T_invertible} and Remark~\ref{rmk:gramian_riccati} for the details on the role of the reachability Gramian. However, this approach fails for $ k\in \bbra{1,k_\rmr-1} $, where $ \calR(k,0) $ can be singular. This issue does not arise in the continuous-time case~\cite{Chen2018}. Instead of $ \calR(k,0) $, our idea is to use $ \calR(\ft,k) $, which is invertible for \gram{$ k \in \bbra{0,k_\rmr} $} to ensure the invertibility of $ X_k $ on the time interval $ \bbra{0,k_\rmr-1} $. 
    In \eqref{eq:controllability}, we will see that the invertibility of $ \calR(\ft,k) $ is equivalent to that of the reachability Gramian of a backward system \eqref{eq:system_inv}, which plays an important role in proving Proposition~\ref{prop:opt_inv}.
    Hence, Assumption~\ref{ass:reachability} assumes the reachability of the forward and backward systems \eqref{eq:system},~\eqref{eq:system_inv} so that the whole time interval $ \bbra{0,\ft} $ is covered by the respective invertible Gramians.

    By the same argument as in \cite[Remark~4]{Ito2022maximum}, for a time-invariant system and weight matrix $ A_k \equiv A$, $B_k \equiv B $, $ \cro_k \equiv \cro $, the invertibility of the reachability Gramians $ \calR(k,0)$, $\calR(\ft,k) $ in Assumption~\ref{ass:reachability} is always satisfied when $ (\bara,B) = (A- B\cro^\top,B) $ is reachable and \gram{$ \ft \ge 2n $}. 
    \rev{Moreover, by using the Popov--Belevitch--Hautus test, it can be easily checked that the reachability of $ (\bara, B) $ is equivalent to that of $ (A,B) $.}
    \hfill $ \diamondsuit $
\end{rmk}
}

\begin{rmk}
    Even if $ Q_k \equiv 0 $, \modi{$ \cro_k \equiv 0 $}, Theorem~1 is novel compared to \cite[Theorem~1]{Ito2022maximum}, which deals with this case. The work~\cite{Ito2022maximum} uses the inverse matrices of $ \Pi_k, H_k $ for the construction of the optimal policy. However, \cite[Remark~5]{Ito2022maximum} gives examples \cyan{where singular matrices $ \Pi_k, H_k $ are required for the density control.} Since Theorem~\ref{thm:opt_policy} does not require the invertibility of $ \Pi_k, H_k $, it can be applied to such cases.
    \hfill $ \diamondsuit $
\end{rmk}

Next, we consider Problem~\ref{prob:density_control} whose density constraints \eqref{eq:cov_constraint} are replaced by the general mean constraints~\eqref{eq:general_mean} (referred to as Problem~\hypertarget{prob_gen}{\ref{prob:density_control}$'$}). For notational simplicity, we again let $ \varepsilon = 1$, $R_k \equiv I $. By the same argument as for $ Q_k \equiv 0 $, \modi{$ \cro_k \equiv 0 $} in \cite{Ito2022maximum}, this problem can be decomposed into mean control and covariance control.
Specifically, the optimal density control is given by $ u_k = \bar{u}_k^* + \check{u}_k^* $. Here, $ \bar{u}^* = \{\bar{u}_k^*\} $ is an optimal mean control of $ \mu_k := \bbE[x_k] $ that solves 
\begin{align}
    &\underset{\bar{u} }{\rm minimize}  && \sum_{k=0}^{\ft-1} \frac{1}{2} \begin{bmatrix}
        \mu_k \\ \bar{u}_k
    \end{bmatrix}^\top \rev{\begin{bmatrix}
        Q_k & \cro_k \\
        \cro_k^\top & I          
    \end{bmatrix}} \begin{bmatrix}
        \mu_k \\ \bar{u}_k
    \end{bmatrix} \label{prob:mean_steering}\\
    &{\rm subject~to} && \mu_{k+1} = A_k \mu_k + B_k \bar{u}_k, \nonumber\\
    & && \mu_0 = \bar{\mu}_0, \ \mu_\ft = \bar{\mu}_\ft , \nonumber
\end{align}
and $ \check{u}_k^* \sim \check{\pi}_k (\cdot | \check{x}_k) $ is the optimal covariance control that steers $ \check{x}_k := x_k - \mu_k $ following $ \check{x}_{k+1} = A_k \check{x}_k + B_k \check{u}_k $ from $ \check{x}_0 \sim \calN(0,\bar{\Sigma}_0) $ to $ \check{x}_N \sim \calN(0,\bar{\Sigma}_\ft)  $ with the cost functional
\begin{align}
    \bbE \Biggl[ \sum_{k=0}^{\ft-1} \biggl( \frac{1}{2}\begin{bmatrix}
        \check{x}_k \\ \check{u}_k
    \end{bmatrix}^\top \begin{bmatrix}
        Q_k & \cro_k \\
        \cro_k^\top & I          
    \end{bmatrix} \begin{bmatrix}
        \check{x}_k \\ \check{u}_k
    \end{bmatrix}  - \calH(\check{\pi}_k (\cdot | \check{x}_k))   \biggr)  \Biggr] . \nonumber
\end{align}

The optimal policy $ \check{\pi}^* $ has already been obtained by Theorem~\ref{thm:opt_policy}.
Necessary conditions for the optimality of $ \bar{u} = \bar{u}^* $ are given by
\begin{align*}
	\mu_{k+1} &= A_k \mu_k + B_k \bar{u}_k^*, \\
	\lambda_k &= Q_k \mu_k + \rev{\cro_k \bar{u}_k^*} + A_k^\top \lambda_{k+1}, \\
	\bar{u}_k^* &= - B_k^\top \lambda_{k+1} - \rev{\cro_k^\top \mu_k} ,\\
	\mu_0 &= \bar{\mu}_0, \ \mu_\ft = \bar{\mu}_\ft ,
\end{align*}
where $ \lambda_k \in \bbR^n $ is a Lagrange multiplier~\cite{Lewis2012}.
By solving the above equations, we obtain
\begin{align}
    &\left[\mu_{k+1}^\top \ \lambda_{k+1}^\top \right]^\top = M_k \left[\mu_{k}^\top \ \lambda_{k}^\top\right]^\top, \nonumber\\
    &\lambda_0 = \phi_{12}^{-1} (\bar{\mu}_\ft - \phi_{11} \bar{\mu}_0), \nonumber\\
    &\bar{u}_k^* = \begin{bmatrix} B_k^\top \bara_k^{-\top} \barq_k - \cro_k^\top & -B_k^\top \bara_k^{-\top} \end{bmatrix} \Phi(k,0) \begin{bmatrix}
        \bar{\mu}_0 \\
        \lambda_0  
    \end{bmatrix} ,\label{eq:opt_mean_steering}
\end{align}
where we assumed the invertibility of $ \phi_{12} $ and \modi{$ \bara_k $} for any $ k \in \bbra{0,\ft-1} $.
Assumption~\ref{ass:reachability} ensures that $ \phi_{12} $ is invertible by Lemma~\ref{lem:T_invertible}.
Since the optimal control problem for $ \mu_k $ is a quadratic programming with equality constraints, \eqref{eq:opt_mean_steering} is the unique optimal solution to \modi{\eqref{prob:mean_steering}.}
Consequently, we arrive at the following result.
\begin{corollary}\label{cor:opt_policy_general}
	Suppose that Assumptions~\ref{ass:invertibility},~\ref{ass:reachability} hold. Then, the unique optimal policy of Problem~\hyperlink{prob_gen}{\ref*{prob:density_control}$ ' $} with $ R_k \equiv I$, $\varepsilon = 1 $ is given by
	\begin{align}
		&\pi_k^* (u|x_k) \nonumber\\
        &= \calN\bigl( u  |  -(I + B_k^\top \Pi_{k+1,+} B_k)^{-1} (B_k^\top \Pi_{k+1,+} A_k + \rev{\cro_k^\top} ) \nonumber\\
        &\qquad\times (x_k- \mu_k^*) + \bar{u}_k^* , ~~ (I + B_k^\top \Pi_{k+1,+} B_k)^{-1}\bigr), \nonumber\\
		&\quad k\in \bbra{0,\ft-1},\ u\in \bbR^m, \ x_k\in \bbR^n , \label{eq:opt_policy_general_mean}
 \end{align}
 where $ \{\Pi_{k,+}\} $ is the solution to \eqref{eq:riccati} with the terminal value \eqref{eq:Pi_terminal}, $ \bar{u}_k^* $ is given by \eqref{eq:opt_mean_steering}, and
 \begin{equation}
	\mu_k^* := 
	\begin{cases}
	\Phi_A (k,0) \bar{\mu}_0 + \sum_{s=0}^{k-1} \Phi_A (k,s+1) B_s \bar{u}_s^* , &\\
	&\hspace{-1.5cm} k\in \bbra{1,\ft-1} , \\
	\bar{\mu}_0, &\hspace{-1.5cm}  k= 0,
	\end{cases} \nonumber
\end{equation}
where $ \Phi_A (k,l) := A_{k-1}A_{k-2} \cdots A_l  $ for $ k > l $ and $ \Phi_A (k,k) = I $.
 \hfill $ \diamondsuit $
\end{corollary}

\cyan{
\begin{rmk}
	Noting that Proposition~\ref{prop:LQ} does not assume the Gaussianity of the initial distribution, the argument for obtaining the policy \eqref{eq:opt_policy_general_mean} still holds when the density constraints \eqref{eq:general_mean} are replaced by the constraints on mean and covariance:
	\begin{align}
		\bbE[x_0] &= \bar{\mu}_0, & \bbE[(x_0 - \bar{\mu}_0)(x_0 - \bar{\mu}_0)^\top] &= \bsigma_0, \label{eq:mean_cov_constraint1}\\
		\bbE[x_\ft] &= \bar{\mu}_\ft, & \bbE[(x_\ft - \bar{\mu}_\ft)(x_\ft - \bar{\mu}_\ft)^\top] &= \bsigma_\ft .\label{eq:mean_cov_constraint2}
 \end{align}
 That is, the policy \eqref{eq:opt_policy_general_mean} is also optimal for Problem~\hyperlink{prob_gen}{\ref*{prob:density_control}$ ' $} with $ R_k \equiv I$, $\varepsilon = 1 $, and \eqref{eq:mean_cov_constraint1},~\eqref{eq:mean_cov_constraint2} instead of the constraints \eqref{eq:general_mean}.
 \hfill $ \diamondsuit $
\end{rmk}
}

\section{Equivalent Backward Density Control Problem}\label{sec:reverse}
\rev{A MaxEnt density control problem of a linear system with $ Q_k \equiv 0 $, $ \cro_k \equiv 0 $ is equivalent to an SB problem~\cite{Ito2022maximum}. In the classical setting, the SB is known to solve the density control problem and its time reversal simultaneously~\cite{Chen2021liaisons}. This section reveals that a similar result holds for our problem with general quadratic cost. Moreover, the established result can be utilized for the proof of Proposition~\ref{prop:positive}.}
Specifically, we investigate the following density control problem of a backward system associated to Problem~\ref{prob:density_control} with $ R_k \equiv I $, $ \varepsilon = 1 $. 
\begin{problem}[\rev{Backward} MaxEnt density control problem]\label{prob:inverse}
	Find a policy $ \varpi = \{\varpi_k\}_{k=\ft-1,\ldots,0} $ that solves
	\begin{align}
			&\underset{\varpi }{\rm minimize}  && \bbE \Biggl[ \sum_{k=1}^{\ft} \biggl( \frac{1}{2} \begin{bmatrix}
                \xi_k \\ v_{k-1}
            \end{bmatrix}^\top \rev{\begin{bmatrix}
                \barq_k & 0 \\
                0 & I          
            \end{bmatrix}} \begin{bmatrix}
                \xi_k \\ v_{k-1}
            \end{bmatrix}  \nonumber\\
			& &&\qquad\qquad - \calH(\varpi_{k-1} (\cdot | \xi_{k}))   \biggr)  \Biggr] \label{eq:cost_inv} \\
			&{\rm subject~to} && \xi_{k} = \rev{\bara_k^{-1}} \xi_{k+1} - \bara_k^{-1}B_k v_k,  \label{eq:system_inv}\\
				& &&v_k \sim \varpi_k (\cdot | \xi_{k+1}) \ {\rm given} \ \xi_{k+1}, \\
				& &&\fin{k = \ft-1,\ldots,0}, \nonumber\\
				& &&\xi_\ft \sim \calN(0,\bar{\Sigma}_\ft), \  \xi_0 \sim \calN(0,\bar{\Sigma}_0) , \label{eq:cov_constraint_inv}
	\end{align}
	where $ \bar{\Sigma}_0, \bar{\Sigma}_\ft \succ 0$, $\modi{\barq_\ft} := 0 $, $ \{\xi_k\} $ is an $ \bbR^n $-valued backward state process, $ \{v_k\} $ is an $ \bbR^m $-valued control process, and a policy $ \varpi_k (\cdot| \xi_{k+1}) $ denotes the conditional density of $ v_k $ given $ \xi_{k+1} $.
	\hfill $ \diamondsuit $
\end{problem}

\rev{Especially when the weight matrix $ \cro_k $ for the cross term is zero, i.e., $ \bara_k = A_k $, $ \barq_k = Q_k $, Problem~\ref{prob:inverse} gives the density control of the time-reversed system $ \xi_k = A_k^{-1} \xi_{k+1} - A_k^{-1} B_k v_k $ of \eqref{eq:system} with the same weight matrices.}
As a straightforward consequence of Proposition~\ref{prop:opt}, the following statement holds for the \rev{backward} problem.
\begin{corollary}\label{cor:reverse}
    Suppose that Assumption~\ref{ass:invertibility} holds and assume that $ \{J_k\}_{k=0}^\ft $ and $ \{P_k\}_{k=0}^\ft $ satisfy the Riccati equations
	\begin{align}
		&J_{k+1} = \modi{\bara_k^{-\top}} J_k \bara_k^{-1} - \bara_k^{-\top} J_k \bara_k^{-1} B_k \nonumber\\
	&\times (I + B_k^\top \bara_k^{-\top} J_k \bara_k^{-1} B_k)^{-1} B_k^\top \bara_k^{-\top} J_k \bara_k^{-1} + \barq_{k+1} , \nonumber\\
    &\quad k\in \bbra{0,\ft-1},\label{eq:riccati_inv}\\
	&P_{k+1} = \bara_k^{-\top} P_k \bara_k^{-1} - \bara_k^{-\top} P_k \bara_k^{-1} B_k \nonumber\\
	&\times(-I + B_k^\top \bara_k^{-\top} P_k \bara_k^{-1} B_k)^{-1} B_k^\top \bara_k^{-\top} P_k \bara_k^{-1} - \barq_{k+1} \nonumber\\
    &\quad k \in \bbra{0,\ft-1}  \label{eq:riccati_inv_dual}
	\end{align}
	with the boundary conditions
	\begin{align}
		\bsigma_0^{-1} =  J_0 + P_0, \ \bsigma_\ft^{-1} =  J_\ft + P_\ft .\label{eq:boundary_inv}
	\end{align}
	Assume further that $ I + B_k^\top \bara_k^{-\top} J_{k} \bara_k^{-1} B_k \succ 0 $ for any $ k\in \bbra{0,\ft-1} $. Then, the unique optimal policy of Problem~\ref{prob:inverse} is given by
	\begin{align}
		 &\varpi_k^* (v|\xi_{k+1}) = \calN\bigl( v  | \ (I + B_k^\top \bara_k^{-\top} J_k \bara_k^{-1} B_k)^{-1} \nonumber\\
         &\times B_k^\top \bara_k^{-\top}  J_{k} \bara_k^{-1} \xi_{k+1}, ~ (I + B_k^\top \bara_k^{-\top} J_{k} \bara_k^{-1} B_k)^{-1}\bigr) \label{eq:opt_policy_reverse}
	\end{align}
	for any $ k \in \bbra{0,\ft-1}, \ v\in \bbR^m, $ and $ \xi_{k+1} \in \bbR^n  $.
   \hfill $ \diamondsuit $
\end{corollary}

The closed-form expression of the solution $ \{J_k,P_k\} $ to \eqref{eq:riccati_inv}--\eqref{eq:boundary_inv} can be obtained \fin{similarly} to Proposition~\ref{thm:riccati_solution}.
The focus of this section is to establish the equivalence between Problems~\ref{prob:density_control},~\ref{prob:inverse}, which is shown in the following proposition.
That is, the optimal policy of Problem~\ref{prob:inverse} can be obtained by solving Problem~\ref{prob:density_control}, and vice versa.
The proof is deferred to Appendix~\ref{app:reverse}.
\begin{proposition}\label{prop:opt_inv}
    Suppose that Assumption~\ref{ass:invertibility} holds and assume that $ \{\Pi_k\}_{k=0}^\ft $ and $ \{H_k\}_{k=0}^\ft $ satisfy the equations \eqref{eq:riccati},~\eqref{eq:H_riccati}, respectively.
	Then, $ \{J_k\}_{k=0}^\ft $ and $ \{P_k\}_{k=0}^\ft $ given by $ J_k := H_k + \barq_k $, $ P_k := \Pi_k - \barq_k $ solve the equations~\eqref{eq:riccati_inv},~\eqref{eq:riccati_inv_dual}, respectively, and it holds that for any $ k \in \bbra{0,\ft-1} $,
    \begin{align}
        I + B_k^\top \Pi_{k+1} B_k &= (I - B_k^\top \bara_k^{-\top} P_k \bara_k^{-1} B_k)^{-1}, \label{eq:Pi_P_posi} \\
        I - B_k^\top H_{k+1} B_k &= (I + B_k^\top \bara_k^{-\top} J_k \bara_k^{-1} B_k)^{-1} .
    \end{align}
    In addition, assume that $ \{\Pi_k\} $ and $ \{H_k\} $ satisfy the boundary conditions~\eqref{eq:boundary}.
    Then, $ \{J_k\} $ and $ \{P_k\}$ satisfy the boundary conditions~\eqref{eq:boundary_inv}.
    Assume further that $ I - B_k^\top H_{k+1} B_k \succ 0 $ for any $ k\in \bbra{0,\ft-1} $.
	Then, the policy
	\begin{align}
		\varpi_k^* (v|\xi_{k+1}) &= \calN\bigl( v  | \  B_k^\top H_{k+1} \xi_{k+1},~ I - B_k^\top  H_{k+1} B_k \bigr), \nonumber\\
		&\quad k\in \bbra{0,\ft-1},\ v\in \bbR^m, \ \xi_{k+1}\in \bbR^n  \label{eq:opt_policy_inv}
 \end{align}
 is the unique optimal policy of Problem~\ref{prob:inverse}.
 Moreover, let $ \{x_k^*\} $ and $ \{\xi_k^*\} $ be the optimal state processes obtained by Proposition~\ref{prop:opt} and \eqref{eq:opt_policy_inv}, respectively. Then, $ \xi_k^* $ has the same distribution as $ x_k^* $, that is,
 \begin{equation}\label{eq:symmetry}
    \xi_k^*  \sim \calN\left(0, \bbE[x_k^* (x_k^*)^\top]\right) , ~~ \forall k\in \bbra{0,\ft} .
 \end{equation}
	\hfill $ \diamondsuit $
\end{proposition}

The converse of the above result also holds. That is, a solution $ \{\Pi_k,H_k\} $ to \eqref{eq:riccati}--\eqref{eq:boundary} can be obtained by a solution $ \{J_k, P_k\} $ to \eqref{eq:riccati_inv}--\eqref{eq:boundary_inv} as $ \Pi_k = P_k + \barq_k$, $H_k = J_k - \barq_k $.
This equivalence is utilized for analyzing the coupled Riccati equations~\eqref{eq:riccati}--\eqref{eq:boundary}; see the proof of Proposition~\ref{prop:positive} (Appendix~\ref{app:positive}).

By Propositions~\ref{prop:positive},~\ref{prop:opt_inv}, we get the explicit form of the optimal policy of Problem~\ref{prob:inverse}.
\begin{corollary}\label{cor:backward}
    Suppose that Assumptions~\ref{ass:invertibility},~\ref{ass:reachability} hold.
    \modi{Let $ \{H_{k,+}\} $ be a solution to \eqref{eq:H_riccati} with the terminal value \eqref{eq:H_terminal}.}
    Then, the unique optimal policy of Problem~\ref{prob:inverse} is given by
    \begin{align}
        &\varpi_k^* (v|\xi_{k+1}) = \calN\bigl( v  | \  B_k^\top H_{k+1,+}  \xi_{k+1},~ I - B_k^\top  H_{k+1,+} B_k \bigr) \label{eq:opt_policy_inv_thm}
    \end{align}
    for any $ k\in \bbra{0,\ft-1} $, $v\in \bbR^m$, and $\xi_{k+1}\in \bbR^n $. \modi{Moreover, \eqref{eq:symmetry} holds.}
     \hfill $ \diamondsuit $
\end{corollary}

\revv{Conversely, the optimal policy \eqref{eq:opt_policy_thm} of the forward problem can be expressed in terms of $ P_k $. Let $ \{P_{k,+}\} $ be a solution to \eqref{eq:riccati_inv_dual} with the initial value $ P_{0,+} = Y_{0,+} - \barq_{0} $. Then, the policy \eqref{eq:opt_policy_thm} can be rewritten as
\begin{align}
    \pi_k^* (u|x_k) &= \calN\bigl( u |  -(\cro_k^\top + B_k^\top \bara_k^{-\top} P_{k,+}) x_k, \nonumber\\
    &\qquad\qquad I - B_k^\top \bara_k^{-\top} P_{k,+} \bara_k^{-1} B_k \bigr) . \label{eq:opt_forward_another}
\end{align}
The derivation is given in Appendix~\ref{app:reverse}.
}

\rev{The classical SB problem is equivalent to the density control of the continuous-time system $ \rmd x(t) = u(t)\rmd t + \rmd w(t) $ with the cost functional $ \bbE[ \int_0^1 \|u(t)\|^2 \rmd t] $, density constraints $ x(0) \sim \rho_0 $, $ x(1) \sim \rho_1 $, and a Wiener process $ w(t) $~\cite{Chen2021liaisons}. The optimal state density is given by $ \rho_{0\rightarrow 1}(t,x) = \psi(t,x) \what{\psi} (t,x) $, $ t\in [0,1] $, where $ \psi $ and $ \what{\psi} $ are solutions to partial differential equations coupled through their boundary conditions, called the \schr~system. Now, let us consider swapping the density constraints as $ x(0) \sim \rho_1 $, $ x(1) \sim \rho_0 $. Then, the new optimal state density is given by $ \rho_{1\rightarrow 0}(t,x) = \what{\psi} (1-t,x) \psi(1-t,x) = \rho_{0\rightarrow 1}(1-t,x) $, which is the time reversal of $ \rho_{0\rightarrow 1} $.
The relationship established in Corollary~\ref{cor:backward} can be seen as the generalization of the above result for the SB to discrete-time linear systems with general quadratic cost. The coupled Riccati equations~\eqref{eq:riccati},~\eqref{eq:H_riccati} (or \eqref{eq:riccati_inv},~\eqref{eq:riccati_inv_dual}) correspond to the \schr~system, and the fact that $ x_k^* $ and $ \xi_k^* $ have the same distribution corresponds to $ \rho_{1\rightarrow 0}(t,x) = \rho_{0\rightarrow 1}(1-t,x) $.
Although a characterization of Problem~\ref{prob:density_control} via an SB is not known for $ Q_k \neq 0 $, $ \cro_k \neq 0 $, the aforementioned relationship between the forward and backward MaxEnt density control problems still holds.}

\section{Unregularized Density Control as the Zero-Noise Limit of MaxEnt Density Control}\label{sec:zero-noise}
In this section, we reveal the relationship between the MaxEnt density control problem (Problem~\ref{prob:density_control}) and the following unregularized density control problem.
\begin{problem}[Unregularized density control problem]\label{prob:density_control_zero}
	Find a control process $ u = \{u_k\}_{k=0}^{\ft-1} $ that solves
	\begin{align}
			&\underset{u \in \calU}{\rm minimize}  && \bbE \Biggl[ \sum_{k=0}^{\ft-1} \frac{1}{2} \begin{bmatrix}
                x_k \\ u_k
            \end{bmatrix}^\top \rev{\begin{bmatrix}
                Q_k & \cro_k \\
                \cro_k^\top & R_k          
            \end{bmatrix}} \begin{bmatrix}
                x_k \\ u_k
            \end{bmatrix}    \Biggr] \label{eq:cost_zero} \\
			&{\rm subject~to} && \eqref{eq:system}, \ \eqref{eq:cov_constraint} ,\nonumber
	\end{align}
	where $ \calU $ denotes the set of all square-summable control processes adapted to $ \{x_k\} $.
	\hfill $ \diamondsuit $
\end{problem}

To investigate the above problem, we first present the optimal policy of Problem~\ref{prob:density_control} for general $ R_k\succ 0 $, $\varepsilon> 0 $. Recall that Problem~\ref{prob:density_control} can be transformed into \eqref{prob:transform}.
Let $ B_{R,k} := B_k R_k^{-1/2} $, $ \Pi_k^{(\varepsilon)} := \varepsilon \Pi_k$, $H_k^{(\varepsilon)} := \varepsilon H_k $. Then, the Riccati equations to be solved associated with \eqref{prob:transform} and the boundary conditions~\eqref{eq:boundary} are given by
\begin{align}
	\Pi_k^{(\varepsilon)} &= \bara_k^\top \Pi_{k+1}^{(\varepsilon)} \bara_k - \bara_k^\top \Pi_{k+1}^{(\varepsilon)} B_{R,k} \nonumber\\
	&\quad\times (I + B_{R,k}^\top \Pi_{k+1}^{(\varepsilon)} B_{R,k})^{-1}  B_{R,k}^\top \Pi_{k+1}^{(\varepsilon)} \bara_k + \barq_k, \nonumber \\
			H_k^{(\varepsilon)} &= \bara_k^\top H_{k+1}^{(\varepsilon)} \bara_k - \bara_k^\top H_{k+1}^{(\varepsilon)} B_{R,k} \nonumber\\
			&\quad\times (-I + B_{R,k}^\top H_{k+1}^{(\varepsilon)} B_{R,k})^{-1}  B_{R,k}^\top H_{k+1}^{(\varepsilon)} \bara_k - \barq_k , \nonumber\\
			\varepsilon \bsigma_0^{-1} &= \Pi_0^{(\varepsilon)} + H_0^{(\varepsilon)}, ~~ \varepsilon \bsigma_\ft^{-1} = \Pi_\ft^{(\varepsilon)} + H_\ft^{(\varepsilon)} ,\nonumber
\end{align}
which take the same form as \eqref{eq:riccati}--\eqref{eq:boundary}.
\fin{Similarly} to \eqref{eq:transition_M}, denote by $ \Phi_{M(R)} (k,l) $ the state-transition matrix for
\begin{equation*}
M_k(R_k) :=
\begin{bmatrix}
	\bara_k + B_k R_k^{-1} B_k^\top \bara_k^{-\top} \barq_k & - B_k R_k^{-1} B_k^\top  \bara_k^{-\top} \\
	- \bara_k^{-\top} \barq_k & \bara_k^{-\top}
\end{bmatrix}.
\end{equation*}
Partition $ \Phi_{M(R)} (k,l) $ as in \eqref{eq:split} and denote the {$ (i,j) $-th} block of $ \Phi_{M(R)} (k,l) $ by $ \Phi_{R,ij}(k,l) $.
Let $ \phi_{R,ij} := \Phi_{R,ij} (\ft,0) $, $ \varphi_{R,ij} := \Phi_{R,ij} (0,\ft) $.
The reachability Gramian for the system~\eqref{eq:system_trans} is written as
\begin{align}
    &\sum_{k = k_0}^{k_1-1} \Phi_\bara(k_1,k+1) \varepsilon B_{R,k} B_{R,k}^\top \Phi_\bara(k_1,k+1)^\top \nonumber\\
    &= \varepsilon\Gamma(k_1,k_0) \diag \left(R_{k_0}^{-1},R_{k_0 + 1}^{-1},\ldots,R_{k_1 -1}^{-1} \right) \Gamma(k_1,k_0)^\top, \label{eq:gramian_R}
\end{align}
where
\begin{align*}
    &\Gamma(k_1,k_0) \\
    &:= [\Phi_\bara (k_1,k_0 + 1)B_{k_0} ~~ \Phi_\bara (k_1,k_0 + 2)B_{k_0 + 1} \ \cdots \ B_{k_1 - 1} ],
\end{align*}
and $ \diag (R_{k_0}^{-1},R_{k_0 + 1}^{-1},\ldots,R_{k_1 -1}^{-1}) $ denotes the block diagonal matrix with diagonal entries $ \{R_k^{-1}\}_{k=k_0}^{k_1 -1} $. Hence, the invertibility of $ \calR(k_1,k_0) = \Gamma (k_1,k_0)\Gamma (k_1,k_0)^\top $ implies that \eqref{eq:gramian_R} is invertible.
Then, the following result follows immediately from Theorem~\ref{thm:opt_policy}.
\begin{corollary}\label{cor:general_opt}
	Suppose that Assumptions~\ref{ass:invertibility},~\ref{ass:reachability} hold. Then, the unique optimal policy of Problem~\ref{prob:density_control} is given by
	\begin{align}
		&\pi_k^* (u|x_k) = \calN\bigl( u  |  -(R_k + B_k^\top \Pi_{k+1}^{(\varepsilon)} B_k)^{-1} \nonumber\\
        &\times (B_k^\top \Pi_{k+1}^{(\varepsilon)} A_k + \cro_k^\top ) x_k, ~ \varepsilon(R_k + B_k^\top \Pi_{k+1}^{(\varepsilon)} B_k)^{-1}\bigr), \nonumber\\
		&\quad k\in \bbra{0,\ft-1},\ u\in \bbR^m, \ x_k\in \bbR^n , \label{eq:opt_policy_general}
 \end{align}
 where $ \{\Pi_{k}^{(\varepsilon)}\} $ is the solution to \eqref{eq:riccati_general} with the terminal value
 \begin{align}
	&\Pi_{\ft} = - \varphi_{R,12}^{-1} \varphi_{R,11} + \frac{\varepsilon}{2} \bsigma_\ft^{-1}\nonumber\\
      &+ \bsigma_\ft^{-1/2} \left( \frac{\varepsilon^2}{4}I + \bsigma_\ft^{1/2} \varphi_{R,12}^{-1} \bsigma_0 \varphi_{R,12}^{-\top} \bsigma_\ft^{1/2}  \right)^{1/2} \bsigma_\ft^{-1/2} . \label{eq:Pi_terminal_general}
 \end{align}
 \hfill $ \diamondsuit $
\end{corollary}

Unlike the case without the final density constraint in Proposition~\ref{prop:LQ}, it is not trivial that the optimal control of Problem~\ref{prob:density_control} converges to the solution to the unregularized problem (Problem~\ref{prob:density_control_zero}) as $ \varepsilon \searrow 0 $ because $ \Pi_k^{(\varepsilon)} $ depends on $ \varepsilon $.
Nevertheless, the following result ensures the convergence of the MaxEnt density control to the unregularized density control as the regularization vanishes.
\begin{theorem}\label{thm:zero_noise}
	Suppose that Assumptions~\ref{ass:invertibility},~\ref{ass:reachability} hold. Then, as $ \varepsilon \searrow 0 $, the policy \eqref{eq:opt_policy_general} converges to the unique optimal control law of Problem~\ref{prob:density_control_zero} given by
	\begin{align}
		u_k &= -(R_k + B_k^\top \Pi_{k+1}^{(0)} B_k)^{-1} (B_k^\top \Pi_{k+1}^{(0)} A_k + \cro_k^\top ) x_k, \nonumber\\
		&\quad k\in \bbra{0,\ft-1}, \ x_k\in \bbR^n , \label{eq:opt_policy_zero}
 \end{align}
 where $ \{\Pi_{k}^{(0)}\} $ is the solution to \eqref{eq:riccati_general} with the terminal value
 \begin{align}
	\Pi_{\ft} &= - \varphi_{R,12}^{-1} \varphi_{R,11}\nonumber\\
      &\quad + \bsigma_\ft^{-1/2} \left( \bsigma_\ft^{1/2} \varphi_{R,12}^{-1} \bsigma_0 \varphi_{R,12}^{-\top} \bsigma_\ft^{1/2}  \right)^{1/2} \bsigma_\ft^{-1/2} . \label{eq:Pi_terminal_zero}
 \end{align}
 \cyan{That is, the mean and covariance of \eqref{eq:opt_policy_general} converge to the right-hand side of \eqref{eq:opt_policy_zero} and $ 0 $ as $ \varepsilon \searrow 0 $, respectively.}
 \hfill $ \diamondsuit $
\end{theorem}
\begin{proof}
	By the same argument as for Problem~\ref{prob:density_control}, if there exists a solution $ \{\Pi_k\} $ to \eqref{eq:riccati_general} satisfying
	\begin{align}
		\Sigma_{k+1} &= \calA_k(\Pi_{k+1}) \Sigma_k \calA_k(\Pi_{k+1})^\top, \label{eq:sigma_zero} \\
	 \Sigma_0 &= \bsigma_0, \ \Sigma_\ft = \bsigma_\ft , \label{eq:cov}
	\end{align}
	then \eqref{eq:opt_policy_zero} with $ \Pi_{k+1}^{(0)} = \Pi_{k+1} $ is the unique optimal control of Problem~\ref{prob:density_control_zero}. 
	Recall that for any $ \varepsilon > 0 $, the solution $ \{\Pi_k^{(\varepsilon) }\} $ to \eqref{eq:riccati_general} given by the terminal value \eqref{eq:Pi_terminal_general} satisfies
	\begin{align}
		\Sigma_{k+1}^{(\varepsilon)} &= \calA_k(\Pi_{k+1}^{(\varepsilon)}) \Sigma_k^{(\varepsilon)} \calA_k(\Pi_{k+1}^{(\varepsilon)})^\top \nonumber\\
		&\quad + \varepsilon B_k (R_k + B_k^\top \Pi_{k+1}^{(\varepsilon)} B_k)^{-1} B_k^\top, \label{eq:cov_eps}\\
		\Sigma_0^{(\varepsilon)} &= \bsigma_0, \ \Sigma_\ft^{(\varepsilon)} = \bsigma_\ft . \nonumber
	\end{align}
	Let $  \esu_{R,k} := - \Phi_{R,12} (0,k)^{-1} \Phi_{R,11} (0,k) $ for $ k \in \bbra{k_\rmr,\ft} $, \cyan{where the invertibility of $ \Phi_{R,12} (0,k) $ for $ k\in \bbra{k_\rmr,\ft} $ is ensured by Assumption~\ref{ass:reachability} and Lemma~\ref{lem:T_invertible}.} Then, similar to \fin{\eqref{eq:BPiB_k_positive}} in Appendix~\ref{app:positive}, we have for any $ k \in \fin{\bbra{k_\rmr , \ft-1}} $,
	\begin{align}
		R_k +B_{k}^\top \Pi_{k+1}^{(0)} B_{k} &\succeq R_k  +B_{k}^\top \esu_{R,k+1} B_{k} \succ 0 .\label{eq:N-1_positive_zero}
	\end{align}
	Hence, $ (R_k + B_k^\top \Pi_{k+1}^{(0)} B_k)^{-1} $ is finite for \fin{$ k \in \bbra{k_\rmr , \ft-1} $}. The same can be shown for \fin{$ k \in \bbra{0 , k_\rmr-1} $}.
	Therefore, the second term of \eqref{eq:cov_eps} goes to zero as $ \varepsilon \searrow 0 $, which \fin{means} $ \{\Pi_k \} = \{\Pi_k^{(0)} \} $ satisfies \eqref{eq:sigma_zero},~\eqref{eq:cov}, and the covariance of \eqref{eq:opt_policy_general} tends to zero as $ \varepsilon \searrow 0 $.
	This completes the proof.
\end{proof}

\rev{
\begin{rmk}\label{rmk:previous}
    For a simple case where $ Q_k \equiv 0$, $\cro_k \equiv 0 $, $ R_k \equiv I $, $ \ft = 1 $, and $ B_0 $ is invertible, we can verify that our solution aligns with that of \cite{Goldshtein2017}.
	In this case, the optimal initial state feedback control given by \cite{Goldshtein2017} is
	\begin{align}
        u_0 &=- (I + B_0^\top \Lambda B_0)^{-1} B_0^\top \Lambda A_0 x_0 , \label{eq:previous}
	\end{align}
	where $ \Lambda $ is the solution to the algebraic Riccati equation:
	\begin{align}
		&B_0^{-\top} B_0^{-1} \bsigma_\ft \Lambda + \Lambda \bsigma_\ft B_0^{-\top} B_0^{-1} + \Lambda \bsigma_\ft \Lambda \nonumber\\
        &+ B_0^{-\top} B_0^{-1} (\bsigma_\ft - A_0 \bsigma_0 A_0^{\top}) B_0^{-\top} B_0^{-1} = 0. \label{eq:algebra_riccati}
	\end{align}
	Moreover, it is straightforward to check that $ \Lambda = \Pi_1^{(0)} $ is the solution to \eqref{eq:algebra_riccati}. Consequently, \eqref{eq:opt_policy_zero} coincides with \eqref{eq:previous}.
    \hfill $ \diamondsuit $
\end{rmk}
}

\section{Example}\label{sec:example}
Finally, we illustrate the obtained results via a numerical example.
The parameters are given by
\begin{align}
	&A_k =
	\begin{bmatrix}
		1.1 & 0.3 \\
		0.1 & 1.2
	\end{bmatrix}, ~
	B_k =
	\begin{bmatrix}
		0.05 \\ 0.1
	\end{bmatrix}, ~ R_k = 1 , ~\cro_k = \begin{bmatrix}
        0.8 \\ 0.4 
    \end{bmatrix},  \forall k, \nonumber\\
	&\rev{\bar{\mu}_0 = \begin{bmatrix} 1 \\ 1 \end{bmatrix}, \ \bar{\mu}_\ft = \begin{bmatrix}
        0 \\ -4
    \end{bmatrix},} \ \bsigma_0 =
	\begin{bmatrix}
		7 & 1\\ 1 & 5
	\end{bmatrix}, \
	\bsigma_\ft = 0.5 I, \nonumber\\
    &\varepsilon = 1, \ \ft = 30, \label{eq:ex_parameter}
\end{align}
which fulfill Assumptions~\ref{ass:invertibility},~\ref{ass:reachability}.
Fig.~\ref{fig:opt_trajectory} describes sample paths of the optimal state process $ \{x_k^*\} $ and the control process $ \{u_k\} $ obtained by \modi{Corollary~\ref{cor:opt_policy_general}} for \modi{$ Q_k \equiv I, 10I $}.
The black ellipses are the so-called $ 3\sigma $ covariance ellipses given by
\begin{equation}\label{eq:ellipse}
	\{ x\in \bbR^2 : (x-\mu_k)^\top \Sigma_{k}^{-1} (x-\mu_k) = 3^2 \} .
\end{equation}
As can be seen, for the large weight \modi{$ Q_k = 10I $}, the optimal policy steers the samples near the origin to reduce the state cost.
Despite the difference in the transient behavior, the state is successfully steered to the same target density in both cases.

\begin{figure}[!t]
	\begin{minipage}[b]{1.0\linewidth}
		\centering
		\includegraphics[keepaspectratio, scale=0.33]
		{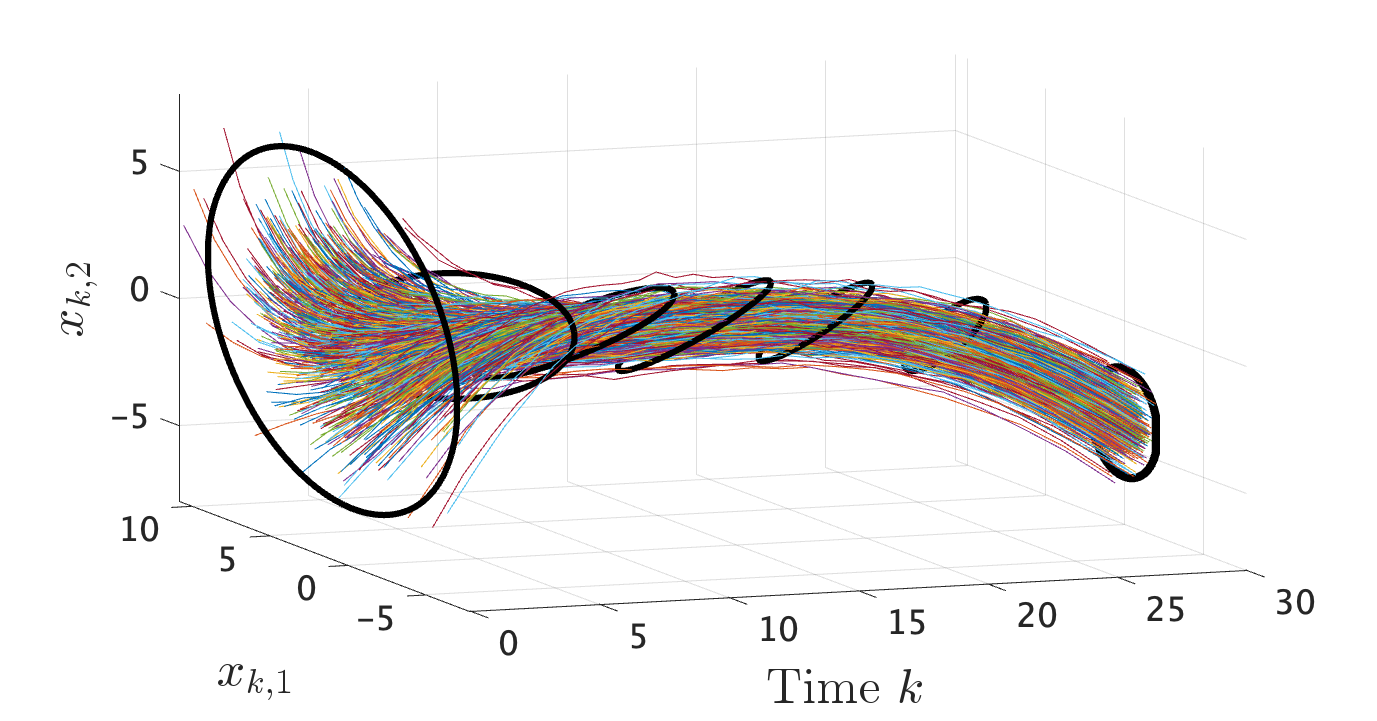}
		\subcaption{$ Q_k = I $}\label{fig:opt_state_Q0}
	\end{minipage}
	\begin{minipage}[b]{1.0\linewidth}
		\centering
		\includegraphics[keepaspectratio, scale=0.33]
		{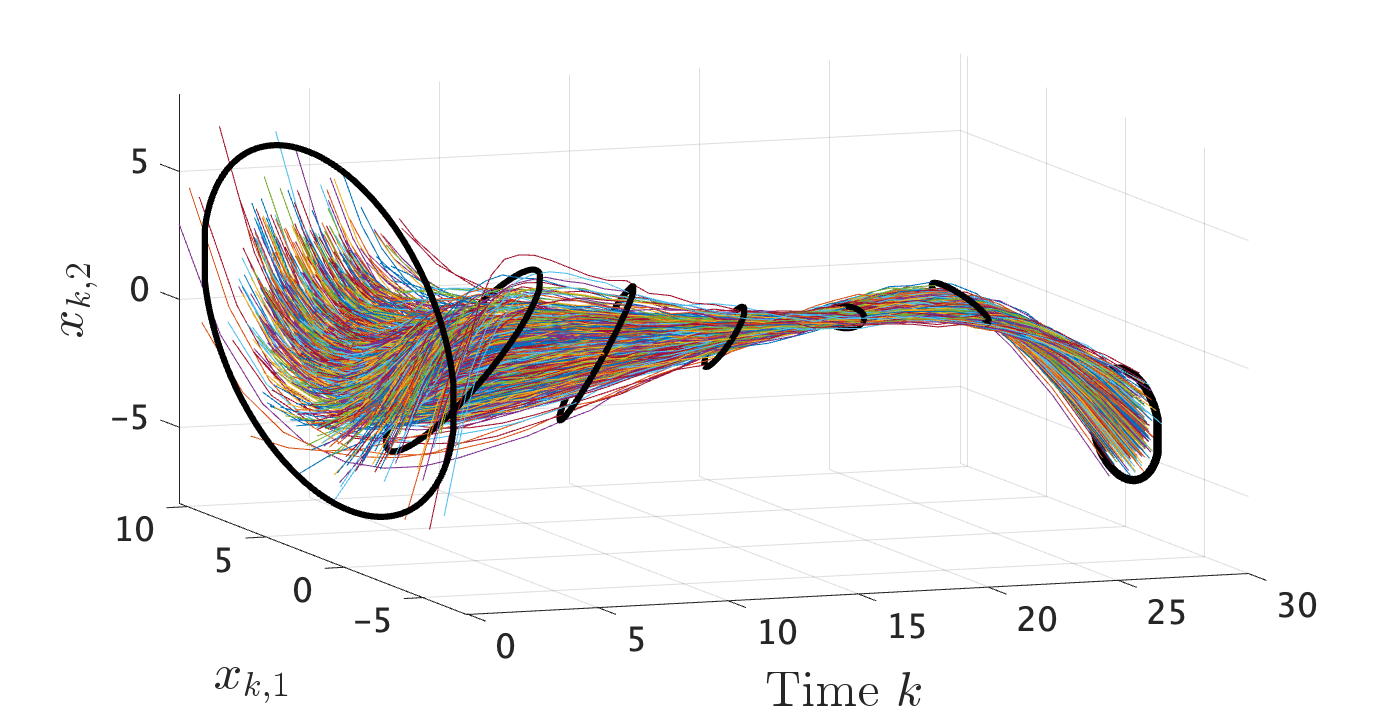}
		\subcaption{$ Q_k = 10I $}\label{fig:opt_state_Q50}
	\end{minipage}

	\vspace{0.5cm}
	\begin{minipage}[b]{0.49\linewidth}
		\centering
		\includegraphics[keepaspectratio, scale=0.24]
		{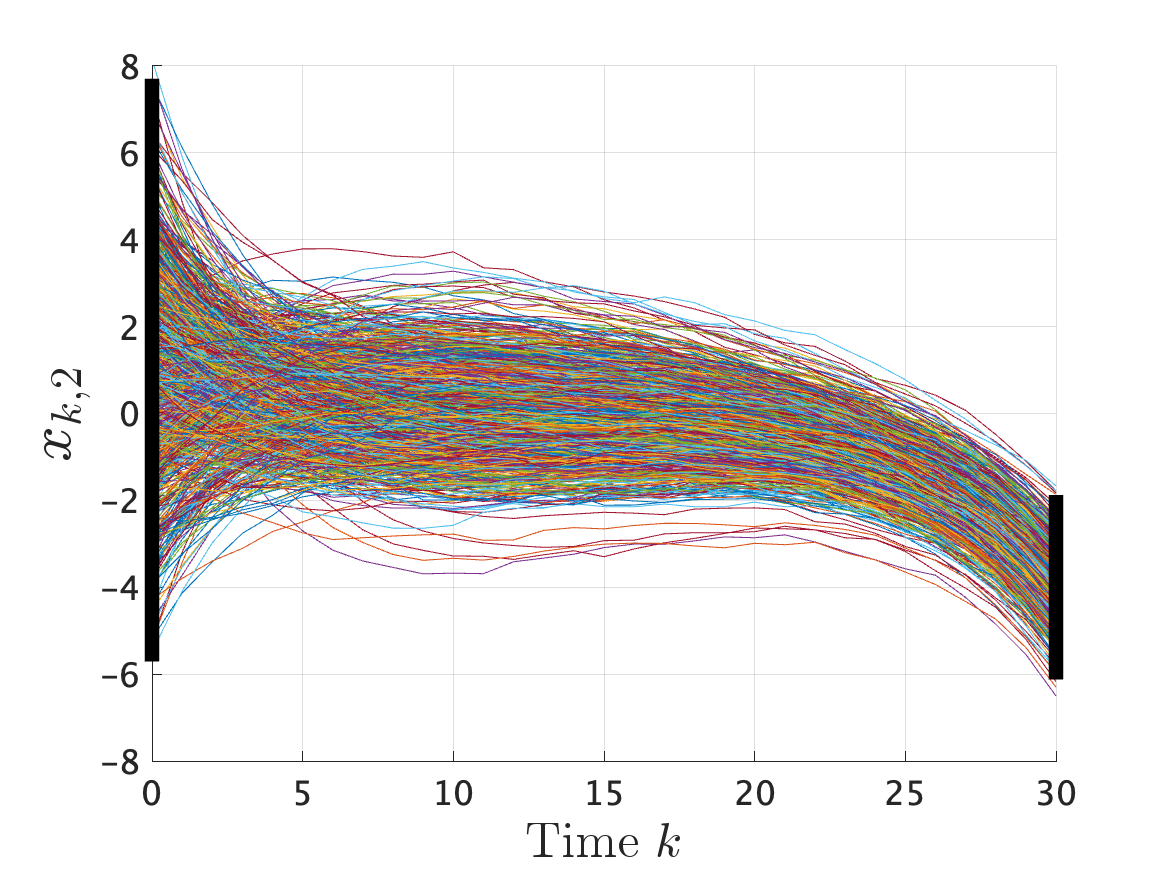}
		\subcaption{$Q_k = I $}\label{fig:opt_state_x2_Q0}
	\end{minipage}
	\begin{minipage}[b]{0.48\linewidth}
		\centering
		\includegraphics[keepaspectratio, scale=0.24]
		{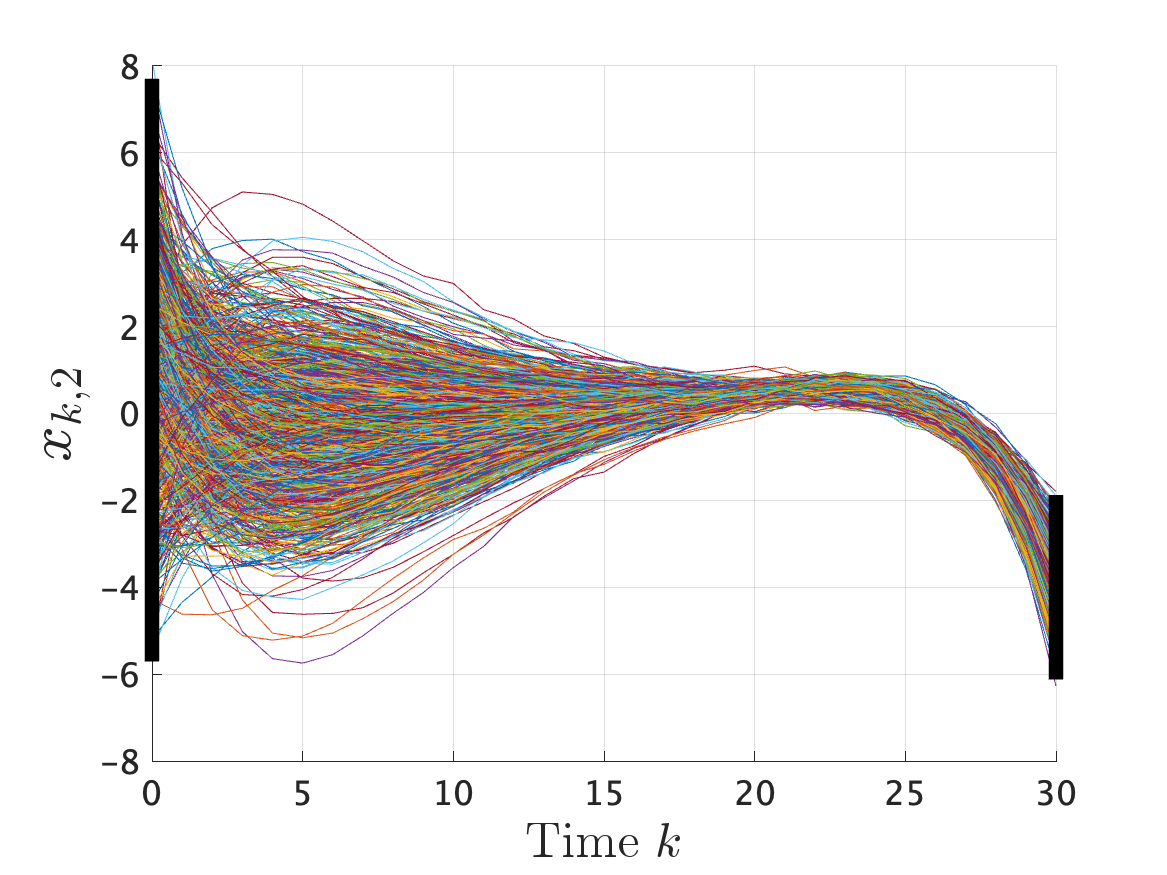}
		\subcaption{$ Q_k = 10I $}\label{fig:opt_state_x2_Q50}
	\end{minipage}

	\vspace{0.5cm}
	\begin{minipage}[b]{0.49\linewidth}
		\centering
		\includegraphics[keepaspectratio, scale=0.24]
		{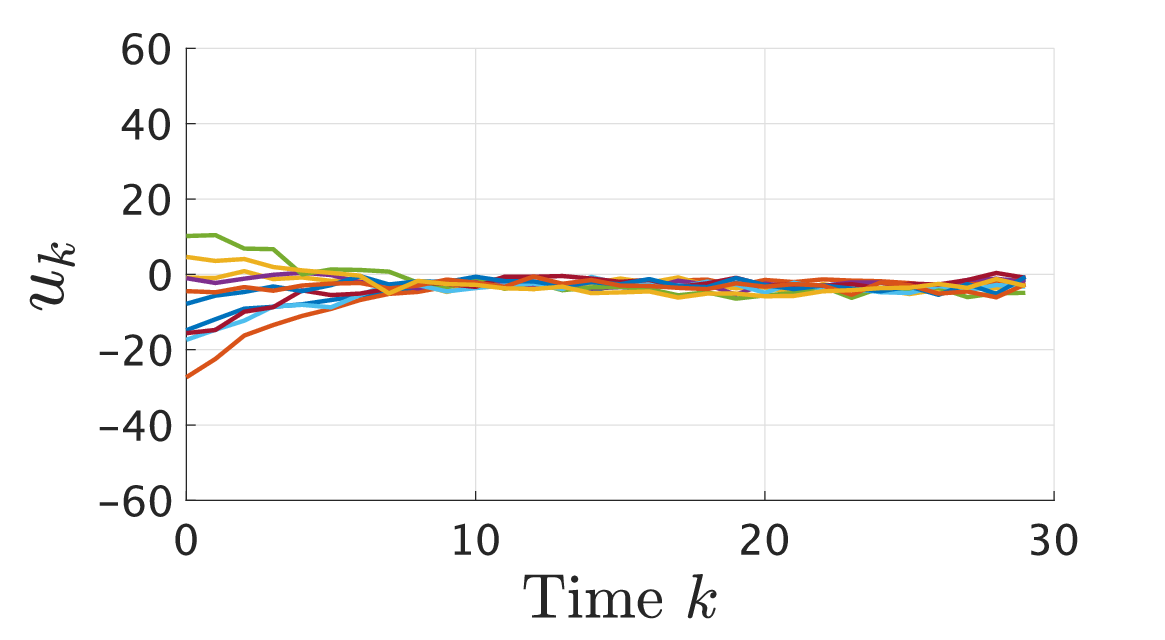}
		\subcaption{$ Q_k = I $}\label{fig:opt_control_Q0}
	\end{minipage}
	\begin{minipage}[b]{0.48\linewidth}
		\centering
		\includegraphics[keepaspectratio, scale=0.24]
		{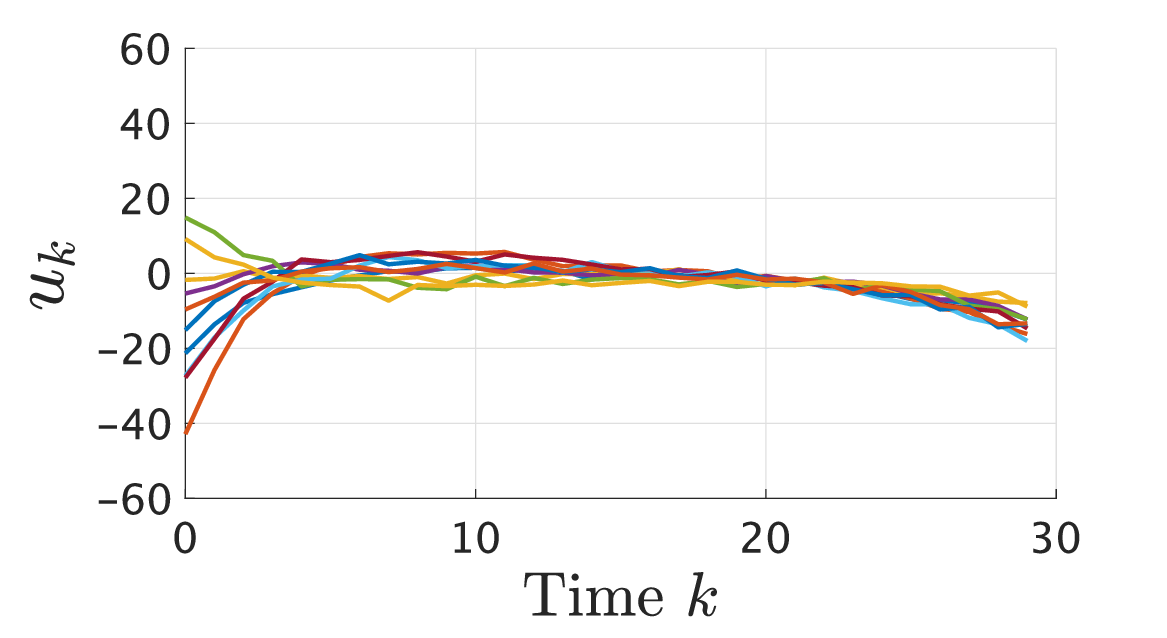}
		\subcaption{$ Q_k = 10I $}\label{fig:opt_control_Q50}
	\end{minipage}
	\caption{\subref{fig:opt_state_Q0}--\subref{fig:opt_state_x2_Q50} $ 1000 $ samples of the optimal state process $ x_k^* = [x_{k,1} \ x_{k,2} ]^\top $ and \subref{fig:opt_control_Q0},~\subref{fig:opt_control_Q50} $ 10 $ samples of the optimal control process for \eqref{eq:ex_parameter} (colored lines) with \modi{$ Q_k \equiv I, 10I $}. The black ellipses are given by \eqref{eq:ellipse}. The black vertical lines in \subref{fig:opt_state_x2_Q0},~\subref{fig:opt_state_x2_Q50} are the $ 3\sigma $ intervals for the initial and target densities.}\label{fig:opt_trajectory}
\end{figure}

Next, we compare the control policies with and without entropy regularization for \eqref{eq:ex_parameter} and $ Q_k \equiv I $, \modi{$ \bar{\mu}_0 = \bar{\mu}_\ft = 0 $}. Fig.~\ref{fig:opt_trajectory_eps} depicts the samples of the optimal state process and the optimal control process for $ \varepsilon = 3,0 $ obtained by Theorems~\ref{thm:opt_policy},~\ref{thm:zero_noise}.
For the regularized case ($ \varepsilon = 3 $), the control process has a large variance resulting in high entropy of the policy.
The same is true for the state process, and each sample path fluctuates.
For the zero-noise case ($ \varepsilon = 0 $), since the optimal policy is deterministic, the samples of the state process show smooth trajectories.

\begin{figure}[!t]
	\begin{minipage}[b]{1.0\linewidth}
		\centering
		\includegraphics[keepaspectratio, scale=0.3]
		{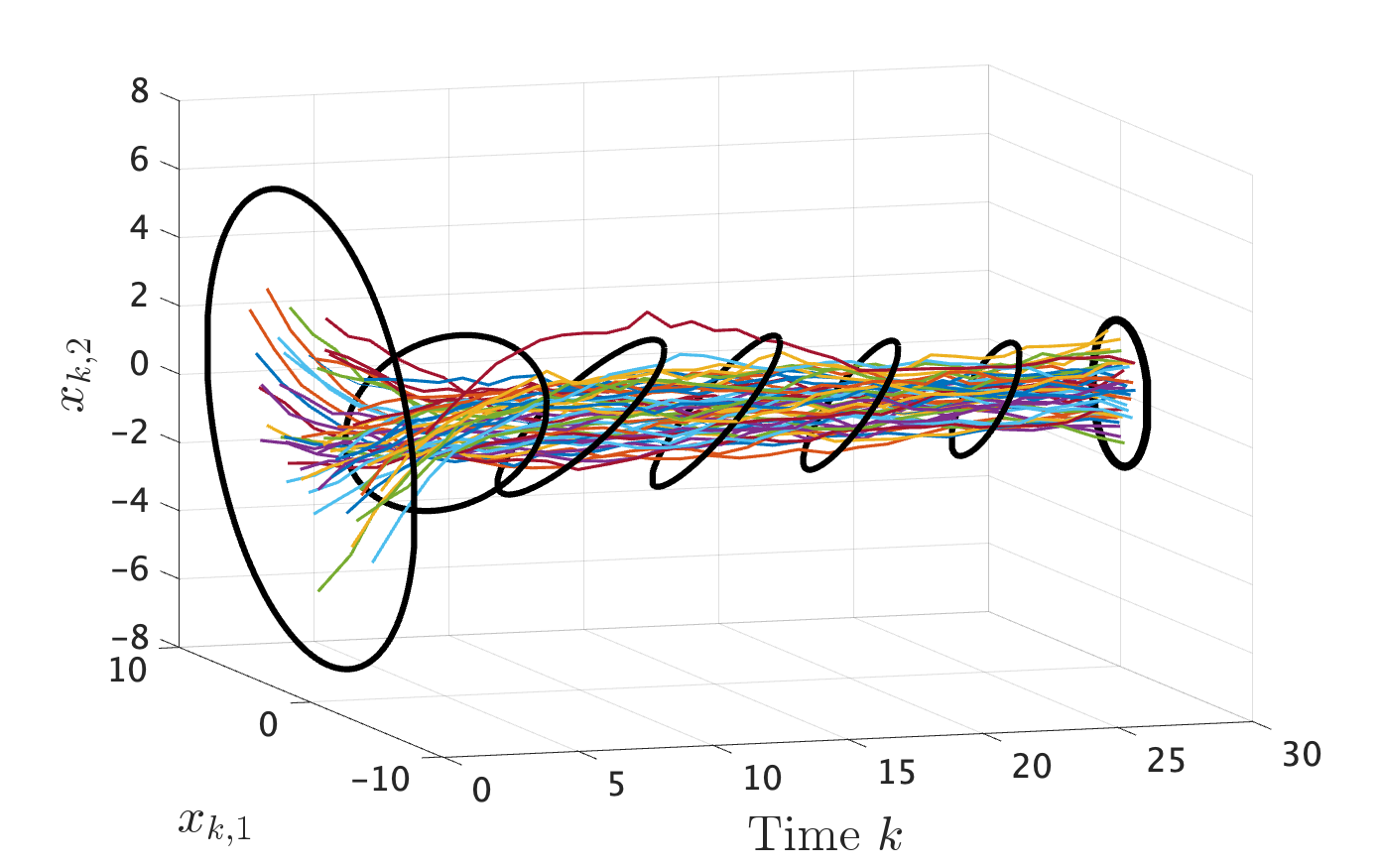}
		\subcaption{$ \varepsilon = 3 $}\label{fig:opt_state_eps3}
	\end{minipage}
	\begin{minipage}[b]{1.0\linewidth}
		\centering
		\includegraphics[keepaspectratio, scale=0.3]
		{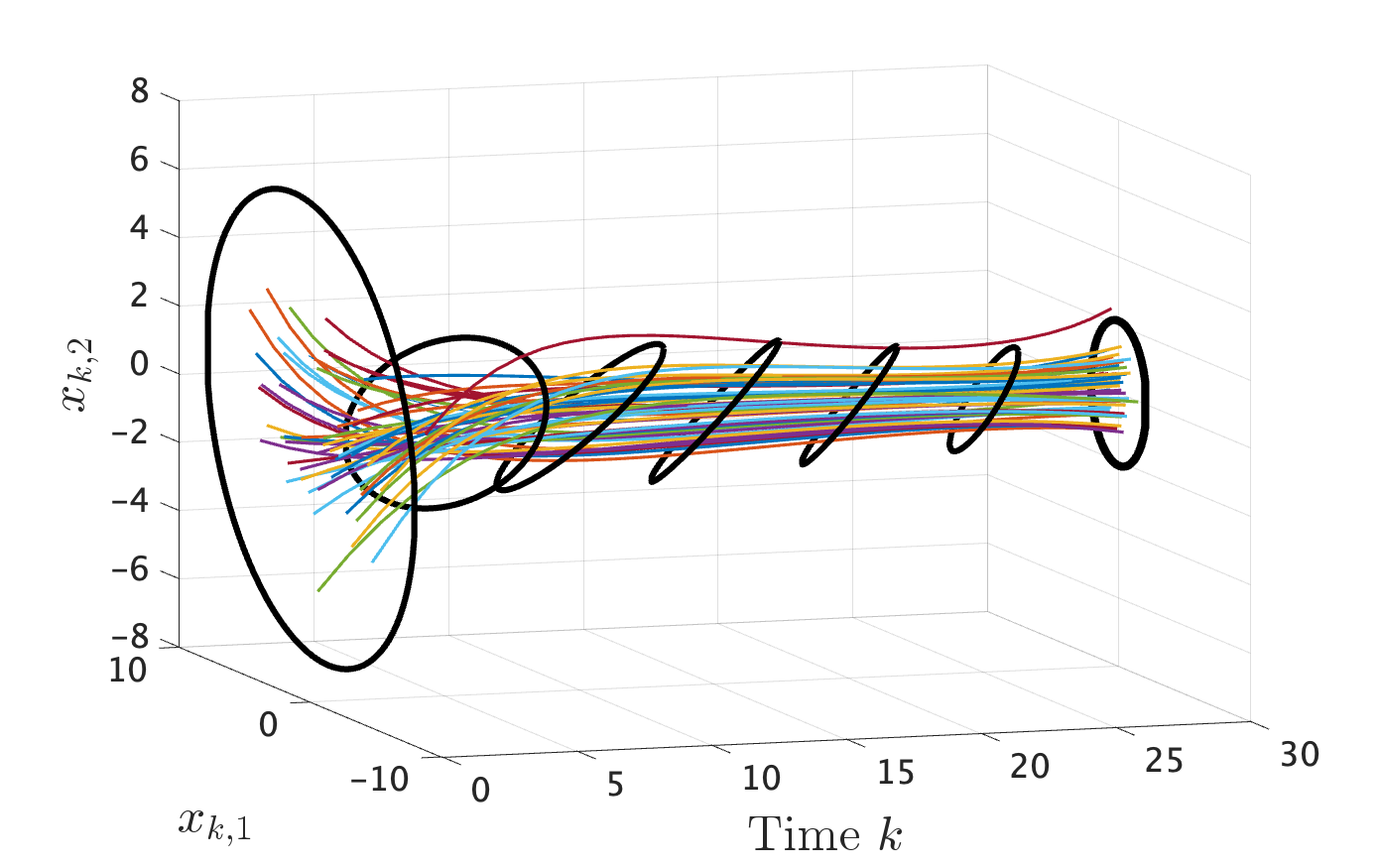}
		\subcaption{$ \varepsilon = 0 $}\label{fig:opt_state_eps0}
	\end{minipage}

	\vspace{0.5cm}
	\begin{minipage}[b]{0.49\linewidth}
		\centering
		\includegraphics[keepaspectratio, scale=0.24]
		{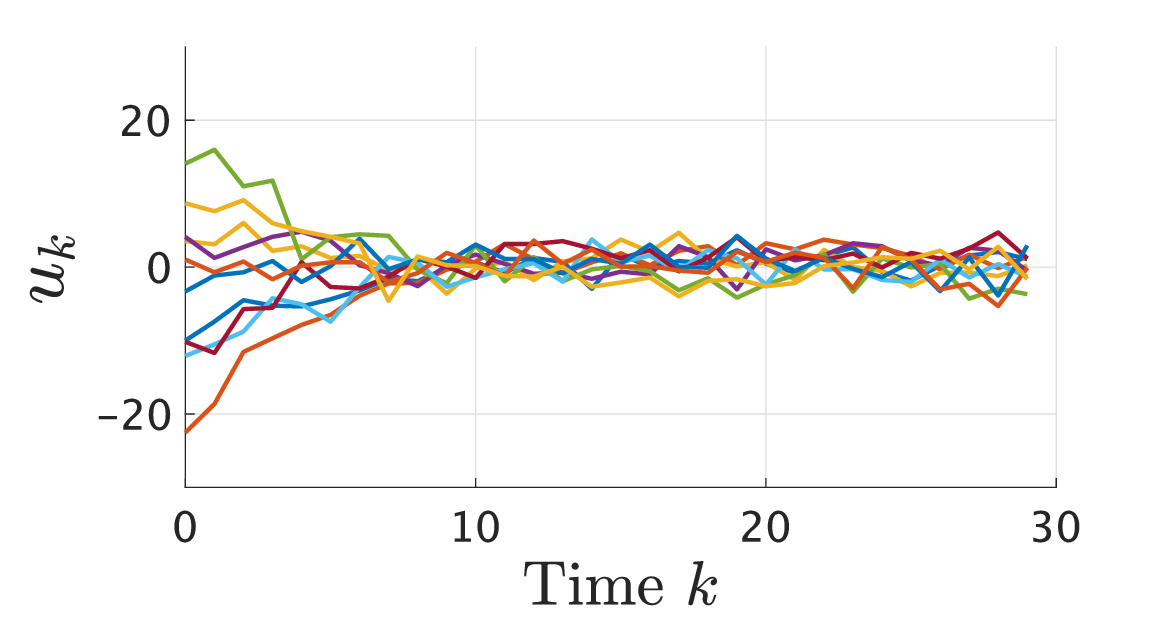}
		\subcaption{$ \varepsilon = 3 $}\label{fig:opt_input_eps3}
	\end{minipage}
	\begin{minipage}[b]{0.48\linewidth}
		\centering
		\includegraphics[keepaspectratio, scale=0.24]
		{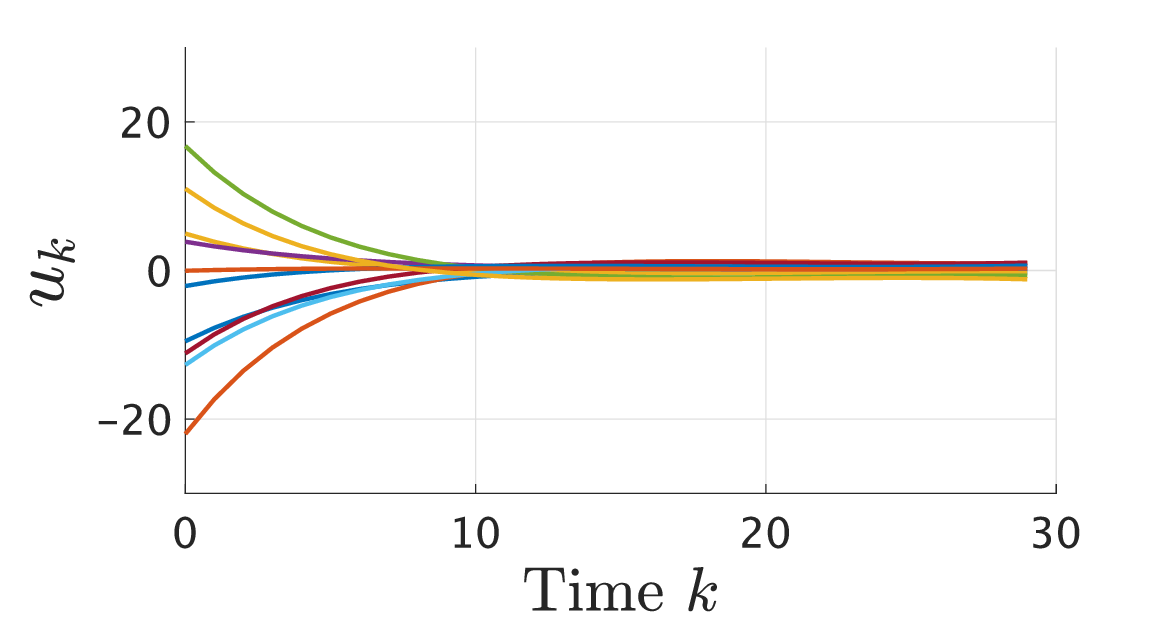}
		\subcaption{$ \varepsilon = 0 $}\label{fig:opt_input_eps0}
	\end{minipage}
	\caption{\subref{fig:opt_state_eps3},~\subref{fig:opt_state_eps0} $ 50 $ samples of the optimal state process $ x_k^* = [x_{k,1} \ x_{k,2} ]^\top $ and \subref{fig:opt_input_eps3},~\subref{fig:opt_input_eps0} $ 10 $ samples of the optimal control process for \eqref{eq:ex_parameter} (colored lines) with $ \varepsilon = 3, 0 $. The black ellipses are given by \eqref{eq:ellipse}.}\label{fig:opt_trajectory_eps}
\end{figure}

\section{Conclusion}\label{sec:conclusion}
In this paper, we studied the MaxEnt density control of \cyan{discrete-time linear deterministic systems} with quadratic cost and Gaussian end-point distributions. Specifically, we showed the existence, uniqueness, and explicit form of the optimal solution by analyzing the coupled Riccati equations. 
Additionally, it was shown that the MaxEnt density control problem has an equivalent backward control problem.
Moreover, by the zero-noise limit where the entropy regularization vanishes, we also obtained the closed-form expression of the optimal state feedback policy of the unregularized density control problem.

An important direction of future work is to study the MaxEnt density control of stochastic systems.
Similarly as mentioned in \cite{Ito2022maximum} for quadratic control cost, our approach based on the coupled Riccati equations will fail to handle stochastic systems. In this case, optimization-based approaches such as in \cite{Bakolas2016,Liu2022} would be suitable. 

\rev{Another direction of future work is to remove or relax the invertibility assumption on $ \bara_k $. The MaxEnt density control without this invertibility can no longer be reduced to solving the coupled Riccati equations~\eqref{eq:riccati},~\eqref{eq:H_riccati}. Hence, it may require a fundamentally different approach.}

\appendices

\section{Proof of Proposition~\ref{prop:riccati}}\label{app:riccati}    % Each appendix must have a short title.
Here, for brevity, we sometimes drop the time index $ k $ when no confusion arises.
   By \eqref{eq:sigma_evolution} \modi{and the invertibility of $ \bara_k $}, it holds that
   \begin{align}
      \Sigma_k^{-1} &= \calA_k(\Pi_{k+1})^\top  (\Sigma_{k+1} - B_k (I + B_k^\top \Pi_{k+1} B_k)^{-1} B_k^\top)^{-1} \nonumber\\
      &\quad \times \calA_k(\Pi_{k+1}) .\label{eq:sig1}
   \end{align}
   Here, we have
   \begin{align*}
      &(\Sigma_{k+1} - B_k (I + B_k^\top \Pi_{k+1} B_k)^{-1} B_k^\top)^{-1} \\
      &= \Sigma_{k+1}^{-1} - \Sigma_{k+1}^{-1} B_k (-I - B_k^\top \Pi_{k+1} B_k + B_k^\top \Sigma_{k+1}^{-1} B_k)^{-1} \\
      &\quad\times B_k^\top \Sigma_{k+1}^{-1} \\
      &=  H_{k+1} + \Pi_{k+1} -  (H_{k+1} + \Pi_{k+1}) B_k \\
      &\quad\times( - I + B_k^\top H_{k+1} B_k )^{-1} B_k^\top (H_{k+1} + \Pi_{k+1}) .
   \end{align*}
   Substituting this into \eqref{eq:sig1} and expanding it yield
   \begin{align*}
      & \Sigma_k^{-1} = \modi{\bara_k^\top}( F_1 + F_2 + F_2^\top + F_3 + F_4 ) \bara_k,
   \end{align*}
   where we defined
   \begin{align*}
      F_1 &:= H + \Pi -  (H + \Pi) B ( - I + B^\top H B )^{-1} B^\top (H + \Pi), \\
      F_2 &:= - \bigl[H + \Pi -  (H + \Pi) B ( - I + B^\top H B )^{-1} \\
      &\quad  \times B^\top (H + \Pi)\bigr] B(I + B^\top \Pi B)^{-1} B^\top \Pi, \\
      F_3 &:= \Pi B (I + B^\top \Pi B)^{-1} B^\top (H + \Pi) B \\
      &\quad \times(I + B^\top \Pi B)^{-1} B^\top \Pi, \\
      F_4 &:= - \Pi B(I + B^\top \Pi B)^{-1} B^\top ( H + \Pi) B(-I + B^\top H B)^{-1} \\
      &\quad \times B^\top (H + \Pi) B(I + B^\top \Pi B)^{-1} B^\top \Pi .
   \end{align*}
   In addition, the above can be rearranged as follows:
   \begin{align*}
      F_2
      &=-(H + \Pi) B(I + B^\top \Pi B)^{-1} B^\top \Pi + (H+\Pi) B \\
      &\quad \times \bigl[ (I+B^\top\Pi B)^{-1} + (-I +B^\top HB)^{-1}  \bigr]  B^\top \Pi \\
      &= (H+\Pi) B (-I + B^\top HB)^{-1} B^\top \Pi , \\
      F_3
      &= \Pi B(I + B^\top \Pi B)^{-1} (B^\top HB - I) \\
      &\quad \times \left[ (I+B^\top\Pi B)^{-1} + (-I +B^\top HB)^{-1}  \right] B^\top \Pi, \\
      F_4 
      &= -\Pi B \left[ (I+B^\top \Pi B)^{-1} + (-I +B^\top HB)^{-1}  \right] \\
      &\hspace{-0.5cm}\times (B^\top HB - I) [ (I+B^\top \Pi B)^{-1} + (-I +B^\top HB)^{-1}  ] B^\top \Pi .
   \end{align*}
   In summary, we get
   \begin{align*}
      \Sigma_k^{-1} &= \bara_k^\top \bigl[ H_{k+1} + \Pi_{k+1} \\
      &\quad- H_{k+1}B_k(-I + B_k^\top H_{k+1} B_k)^{-1} B_k^\top H_{k+1} \\
      &\quad  - \Pi_{k+1} B_k(I + B_k^\top \Pi_{k+1} B_k)^{-1} B_k^\top \Pi_{k+1} \bigr] \bara_k .
   \end{align*}
   Finally, by \eqref{eq:riccati}, we arrive at the Riccati equation~\eqref{eq:H_riccati}.

\section{Proof of Proposition~\ref{prop:linear}}\label{app:linear}
By \eqref{eq:XY} and noting that
\begin{equation}
    M_k = 
    \begin{bmatrix}
		I & G_k \\
		0 & \bara_k^\top
	\end{bmatrix}^{-1}
	\begin{bmatrix}
		\bara_k & 0 \\
		-\barq_k & I
	\end{bmatrix} ,
\end{equation}
we obtain for $ k\in \bbra{0,\ft-1} $,
\begin{align}
	&X_{k+1}  + G_k Y_{k+1} = \bara_k X_k, \label{eq:XY_1} \\
	&\bara_k^\top Y_{k+1} = -\barq_k X_k + Y_k . \label{eq:XY_2}
\end{align}
Assume that for some $ \bar{k} \in \bbra{1,\ft} $, $ X_{\bar{k}} $ is invertible and $ \Pi_{\bar{k}} = Y_{\bar{k}} X_{\bar{k}}^{-1} $ holds. By \eqref{eq:XY_1} and \cyan{the Sherman--Morrison--Woodbury formula~\cite{Horn2012}}, $ X_{\bar{k}-1}^{-1} $ is formally
\begin{align}
	&X_{\bar{k}-1}^{-1} = X_{\bar{k}}^{-1} ( I + G_{\bar{k}-1} Y_{\bar{k}} X_{\bar{k}}^{-1})^{-1} \bara_{\bar{k}-1} \nonumber\\
	&= X_{\bar{k}}^{-1} \left( I - B_{\bar{k}-1} (I + B_{\bar{k}-1}^\top \Pi_{\bar{k}} B_{\bar{k}-1})^{-1} B_{\bar{k}-1}^\top \Pi_{\bar{k}}  \right) \bara_{\bar{k}-1} . \nonumber
\end{align}
Since $ \{\Pi_{k}\} $ satisfies \eqref{eq:riccati}, $ I + B_{\bar{k}-1}^\top \Pi_{\bar{k}} B_{\bar{k}-1} $ is invertible, and thus, $ X_{\bar{k}-1} $ is invertible.
In addition, \eqref{eq:XY_2} yields
\begin{align}
	&Y_{\bar{k}-1} X_{\bar{k}-1}^{-1} = \bara_{\bar{k}-1}^\top Y_{\bar{k}} X_{\bar{k}-1}^{-1} + \barq_{\bar{k}-1} \nonumber\\
	&= \bara_{\bar{k}-1}^\top \Pi_{\bar{k}} \left( I - B_{\bar{k}-1} (I + B_{\bar{k}-1}^\top \Pi_{\bar{k}} B_{\bar{k}-1})^{-1} B_{\bar{k}-1}^\top \Pi_{\bar{k}}  \right) \bara_{\bar{k}-1}  \nonumber\\
    &\quad + \barq_{\bar{k}-1} \nonumber\\
    &= \Pi_{\bar{k}-1} . \label{eq:linear_2}
\end{align}
Since for $ \bar{k} = \ft $, $ X_{\bar{k}} $ is invertible and $ \Pi_{\bar{k}} = Y_{\bar{k}} X_{\bar{k}}^{-1} $, by induction, it holds that $ \Pi_k = Y_k X_k^{-1} $ for any $ k \in \bbra{0,\ft} $.
The same argument as above shows $ H_k = - \what{X}_k^{-\top} \what{Y}_k^\top $.
\cyan{The statement ii) follows immediately from the invertibility assumption on $ X_k, \what{X}_k $ and \fin{a similar argument as in} \eqref{eq:linear_2}.}

   \section{Lemmas for the Proofs of Propositions~\ref{thm:riccati_solution},~\ref{prop:positive}}
   Here, we introduce some lemmas for proving Propositions~\ref{thm:riccati_solution},~\ref{prop:positive}.
   
   \begin{lemma}\label{lemma:phi11_inv}
       Suppose that Assumption~\ref{ass:invertibility} holds.
       Then, for any $ k,l\in \bbra{0,\ft} $, $ \Phi_{11}(k,l) $ is invertible.
       In addition, for $ T_{k,l} := \Phi_{11}(k,l)^{-1} \Phi_{12} (k,l) $, it holds that
       \begin{align}\label{eq:T}
           T_{k,l} 
           \begin{cases}
           \succeq 0, & k < l , \\
           = 0, & k = l, \\
           \preceq 0, & k > l.
           \end{cases}
       \end{align}
       \hfill $ \diamondsuit $
   \end{lemma}
   \begin{proof}
       It is known that $ \Phi_{11} (k,l) $ is invertible for $ k,l \in \bbra{0,\ft} $ by \cite[Lemma~4]{Liu2022}.
       By using $ \Phi_M (k,l+1) = \Phi_M (k,l) M_l^{-1} $, we obtain
       \begin{align}
           &T_{k,l+1} = (\Phi_{11}(k,l) \bara_l^{-1} + \Phi_{12}(k,l)\barq_l \bara_l^{-1})^{-1} \nonumber\\
           &\quad\times (\Phi_{11}(k,l) \bara_l^{-1}G_l + \Phi_{12}(k,l) \barq_l\bara_l^{-1} G_l + \Phi_{12}(k,l) \bara_l^\top) \nonumber\\
           &= G_l + \bara_l (I + T_{k,l} \barq_l)^{-1} T_{k,l} \bara_l^\top . \label{eq:riccati_T2}
       \end{align}
    Since $ \barq_l \succeq 0 $, there exists a matrix $ L_l $ such that $ \barq_l = L_lL_l^\top $, and we get the following Riccati equation:
       \begin{align}
           &T_{k,l+1} = G_l + \bara_l\left( I - T_{k,l} L_l (I + L_l^\top T_{k,l} L_l)^{-1} L_l^\top  \right) T_{k,l} \bara_l^\top .\label{eq:riccati_T}
       \end{align}
       In addition, by \cite[Section~IV]{Chan1984}, \eqref{eq:riccati_T} can be written as
       \begin{align}
           T_{k,l+1} =G_{l} + (\bara_{l} - W_{l} L_{l}^\top) T_{k,l}(\bara_{l} - W_{l} L_{l}^\top)^\top + W_{l} W_{l}^\top  , \label{eq:riccati_K}
       \end{align}
       where $ W_{l} := \bara_{l} T_{k,l} L_{l}\left(I + L_{l}^\top T_{k,l}L_{l} \right)^{-1} $. Here,
       \begin{align*}
           \bara_{l} - W_{l} L_{l}^\top = \bara_{l} (I + T_{k,l} \barq_{l})^{-1}
       \end{align*}
       is invertible.
       Therefore, by \eqref{eq:riccati_K} and $ T_{k,k} = 0 $, we obtain $ T_{k,l} \succeq 0 $ for $ k < l $ and $ T_{k,l} \preceq 0 $ for $ k > l $.
   \end{proof}

   \begin{lemma}[{\cite[Corollary~2]{Liu2022}}]\label{lem:T_invertible}
       Suppose that Assumptions~\ref{ass:invertibility},~\ref{ass:reachability} hold.
       Then, $ \Phi_{12} (k,0) $, $\Phi_{12} (0,k) $ are invertible for any $ k\in \bbra{k_\rmr,\ft} $, and $ \Phi_{12} (k,\ft) $, $\Phi_{12} (\ft,k) $ are invertible for any $ k\in \gram{\bbra{0,k_\rmr}} $. 
       \hfill $ \diamondsuit $
   \end{lemma}

   \rev{
   \begin{rmk}\label{rmk:gramian_riccati}
   The invertibility of the reachability Gramian is essential for this lemma. To see this, it is more convenient to prove Lemma~\ref{lem:T_invertible} based on the Riccati equation~\eqref{eq:riccati_T} rather than the approach in \cite{Liu2022}. When $ Q_k \equiv 0 $, $ \cro_k \equiv 0 $, i.e., $ L_k \equiv 0 $, the solution $ T_{0,l} = \Phi_{11} (0,l)^{-1} \Phi_{12} (0,l) $ to \eqref{eq:riccati_T} coincides with $ \calR(l,0) $. Hence, the invertibility of $ \calR(k,0) = \Phi_{11} (0,k)^{-1} \Phi_{12} (0,k) $ implies that $ \Phi_{12} (0,k) $ is invertible. Even when $ Q_k \neq 0 $, $ \cro_k \neq 0 $, by the same argument as for \cite[Lemma~7.3]{Jazwinski1970}, we can show that the solution $ T_{0,l} $ to the Riccati equation \eqref{eq:riccati_T} with the initial condition $ T_{0,0} = 0 $ is invertible under the invertibility of $ \calR(l,0) $. Similarly, the invertibility of $ \Phi_{12} (k,0) $, $ \Phi_{12} (k,\ft) $, $ \Phi_{12} (\ft,k) $ can be shown via the Riccati equation and the invertible reachability Gramian.
   \hfill $ \diamondsuit $
   \end{rmk}
   }
   
   \begin{lemma}\label{lem:T_monotone}
       Suppose that Assumption~\ref{ass:invertibility} holds.
       Then, for any $ k \in \bbra{1,\ft}, \ l\in \bbra{0,\ft} $, it holds that
       \begin{equation}\label{eq:T_monotone}
           T_{k,l}  \preceq T_{k-1,l}  .
       \end{equation}
       Especially when $ T_{k,l} $ and $T_{k-1,l} $ are invertible, it holds that
       \begin{equation}\label{eq:Tinv_monotone}
           T_{k,l}^{-1} \succeq T_{k-1,l}^{-1} .
       \end{equation}
       \hfill $ \diamondsuit $
   \end{lemma}
   \begin{proof}
       To simplify notations, we sometimes omit the variable $ l $ of $ \Phi_{11} (k,l) $ as $ \Phi_{11}(k) $. We further drop the time index $ k $ when no confusion can arise.
           By using \eqref{eq:M_inv} and
           \begin{align}
                   &\Phi_M (k-1,l)
               = M_{k-1}^{-1} \Phi_M (k,l), \label{eq:Phi_k-1_k}
           \end{align}
           we obtain
           \begin{align*}
               &T_{k,l} - T_{k-1,l} \\
               &= \Phi_{11} (k)^{-1} \Phi_{12}(k) - (\Phi_{11} (k) + G_{k-1} \Phi_{21}(k))^{-1} \\
               &\quad \times (\Phi_{12} (k) + G_{k-1} \Phi_{22}(k)) .
           \end{align*}
           Here, it holds that
           \begin{align*}
               &(\Phi_{11} (k) + G \Phi_{21}(k))^{-1} \\
               &= \Phi_{11} (k)^{-1} - \Phi_{11} (k)^{-1} B ( I + B^\top \Phi_{21}(k) \Phi_{11} (k)^{-1}  B)^{-1} \\
               &\quad\times B^\top \Phi_{21} (k)\Phi_{11} (k)^{-1} .
           \end{align*}
           Therefore, we get
           \begin{align}
               &T_{k,l} - T_{k-1,l} \nonumber\\
               &= \Phi_{11}(k)^{-1} B ( I + B^\top \Phi_{21}(k)\Phi_{11}(k)^{-1}B)^{-1} B^\top \nonumber\\
               &\quad\times \Phi_{21}(k) \Phi_{11}(k)^{-1} \Phi_{12}(k) - \Phi_{11}(k)^{-1} BB^\top \Phi_{22}(k) \nonumber\\
               &\quad + \Phi_{11}(k)^{-1} B ( I + B^\top \Phi_{21}(k)\Phi_{11}(k)^{-1}B)^{-1} B^\top \Phi_{21}(k) \nonumber\\
               &\quad \times \Phi_{11}(k)^{-1} BB^\top \Phi_{22} (k) . \label{eq:Tk_Tk_1}
           \end{align}
           For the second term of \eqref{eq:Tk_Tk_1}, we have
           \begin{align*}
               &\Phi_{11}(k)^{-1} BB^\top \Phi_{22}(k) \\
               &= \Phi_{11}(k)^{-1} B(I + B^\top \Phi_{21}(k) \Phi_{11}(k)^{-1}B)^{-1}\\
               &\quad \times (I + B^\top \Phi_{21}(k) \Phi_{11}(k)^{-1}B) B^\top \Phi_{22}(k) .
           \end{align*}
           Substituting this into \eqref{eq:Tk_Tk_1} yields
           \begin{align*}
               &T_{k,l} - T_{k-1,l} \\
               &=\Phi_{11}(k)^{-1} B (I + B^\top \Phi_{21}(k)\Phi_{11}(k)^{-1} B)^{-1} B^\top \Phi_{21}(k) \\
               &\quad \times \Phi_{12}(k)^\top \Phi_{11}(k)^{-\top} \\
               &\quad - \Phi_{11}(k)^{-1} B( I + B^\top \Phi_{21}(k)\Phi_{11}(k)^{-1} B)^{-1}  B^\top \Phi_{22}(k) \\
               &= \Phi_{11}(k)^{-1} B( I + B^\top \Phi_{21}(k)\Phi_{11}(k)^{-1} B)^{-1}  B^\top \\
               &\quad \times  (\Phi_{21}(k) \Phi_{12}(k)^\top \Phi_{11}(k)^{-\top} - \Phi_{22}(k)) .
           \end{align*}
           Here, we used the fact that $ \Phi_{11}(k)^{-1} \Phi_{12}(k) $ is symmetric shown in \cite[Lemma~4]{Liu2022}, which also proves $ \Phi_{11}(k) \Phi_{22}(k)^\top - \Phi_{12}(k) \Phi_{21}(k)^\top = I $.
           Consequently, we obtain
           \begin{align*}
               &T_{k,l} - T_{k-1,l} = - \Phi_{11}(k,l)^{-1} B_{k-1} \nonumber\\
               &\times ( I + B_{k-1}^\top \Phi_{21}(k,l)\Phi_{11}(k,l)^{-1} B_{k-1})^{-1} B_{k-1}^\top \Phi_{11}(k,l)^{-\top} .\label{eq:Tk_Tk-1}
           \end{align*}
       
           Therefore, to verify \eqref{eq:T_monotone}, it suffices to show
           \begin{equation}\label{eq:K_positive}
           I + B_{k-1}^\top \Phi_{21}(k,l) \Phi_{11}(k,l)^{-1} B_{k-1} \succ 0. 
           \end{equation}
           Let $ K_{k,l} := \Phi_{21}(k,l) \Phi_{11}(k,l)^{-1} $.
           By \eqref{eq:Phi_k-1_k}, \rev{similar to \eqref{eq:riccati_T2},} for any $ k\in \bbra{1,\ft}, \ l\in \bbra{0,\ft} $, we have
           \begin{align}
               &K_{k-1,l}
               = \barq_{k-1} + \bara_{k-1}^\top K_{k,l} \bara_{k-1} \nonumber\\
               &- \bara_{k-1}^\top K_{k,l} B_{k-1} (I + B_{k-1}^\top K_{k,l} B_{k-1})^{-1} B_{k-1}^\top K_{k,l} \bara_{k-1} , \label{eq:K_riccati} \\
               &K_{l,l} = 0. \label{eq:initlal_K}
           \end{align}
           For $ k \le l $, it is well-known that the Riccati equation \eqref{eq:K_riccati} with the initial value \eqref{eq:initlal_K} satisfies $ I + B_{k-1}^\top K_{k,l} B_{k-1} \succ 0 $, which means \eqref{eq:K_positive} for $ k \le l $.
           Next, we consider the case when $ k > l $.
           Similar to \eqref{eq:Pi_P_positive} in Appendix~\ref{app:reverse}, we have
           \begin{align*}
               &I + B_{k-1}^\top K_{k,l} B_{k-1} \\
               &= (I + B_{k-1}^\top \bara_{k-1}^{-\top} (\barq_{k-1} - K_{k-1,l})\bara_{k-1}^{-1}B_{k-1})^{-1} .
           \end{align*}
           In addition, by \eqref{eq:K_riccati},~\eqref{eq:initlal_K}, and the same proof as in Lemma~\ref{lemma:phi11_inv}, it can be shown that $ K_{k-1,l} \preceq 0 $ for $ k > l $, which implies \eqref{eq:K_positive} for $ k > l $. As a result, we obtain \eqref{eq:T_monotone}.
           Finally, by \eqref{eq:T} in Lemma~\ref{lemma:phi11_inv} and the invertibility assumption on $ T_{k,l}, T_{k-1,l} $, we arrive at \eqref{eq:Tinv_monotone}.
   \end{proof}

   \section{Proof of Proposition~\ref{thm:riccati_solution}}\label{app:solution_riccati}
   \cyan{For finding solutions $ \{\Pi_k,H_k\} $ to \eqref{eq:riccati}--\eqref{eq:boundary}, we construct terminal values $ (X_\ft,Y_\ft) $, $ (\what{X}_\ft, \what{Y}_\ft ) $ such that for any $ k\in \bbra{0,\ft} $, the resulting solutions $ X_k $ and $ \what{X}_k $ to \eqref{eq:XY},~\eqref{eq:XYhat} are invertible, and the boundary conditions~\eqref{eq:boundary2_1},~\eqref{eq:boundary2_2} are satisfied. Then, by Proposition~\ref{prop:linear}, $ \Pi_k = Y_kX_k^{-1} $ and $ H_k = - \what{X}_k^{-\top} \what{Y}_k^\top $ satisfy \eqref{eq:riccati}--\eqref{eq:boundary}.}
   Without loss of generality, we assume that $ X_\ft = \what{X}_\ft  = I $ because their terminal values can be absorbed into $ Y_\ft $ and $ \what{Y}_\ft $ without changing the values of $ \Pi_\ft = Y_\ft X_\ft^{-1} $ and $ H_\ft = -\what{X}_\ft^{-\top} \what{Y}_\ft^\top $. Note that in this case, $ Y_\ft $ and $ \what{Y}_\ft $ are symmetric.
   
   \rev{By the invertibility of $ \varphi_{12} $ (Lemma~\ref{lem:T_invertible}) and using exactly the same argument as in the proof of \cite[Theorem~1]{Chen2018}, we can show that under $ X_\ft = \what{X}_\ft = I $, only the following two sets of terminal values $ Y_\ft $, $ \what{Y}_\ft $ lead to solutions to \eqref{eq:XY},~\eqref{eq:XYhat} satisfying the boundary conditions~\eqref{eq:boundary2_1},~\eqref{eq:boundary2_2}:}
   \begin{align}
   (X_\ft,Y_\ft) &= (I,Y_{\ft,\pm}) = (I,Z_{\ft,\pm} + \bsigma_\ft^{-1}/2) , \label{eq:terminal1}\\
(\what{X}_\ft, \what{Y}_\ft) &= (I,\what{Y}_{\ft,\pm}) := (I,Z_{\ft,\pm} - \bsigma_\ft^{-1}/2) , \label{eq:terminal2}
   \end{align}
   where
   \begin{align}
    &Z_{\ft,\pm} := - \varphi_{12}^{-1} \varphi_{11} \nonumber\\
   &\pm \bsigma_\ft^{-1/2} \left( \frac{1}{4}I + \bsigma_\ft^{1/2} \varphi_{12}^{-1} \bsigma_0 \varphi_{12}^{-\top} \bsigma_\ft^{1/2}  \right)^{1/2} \bsigma_\ft^{-1/2} . \label{eq:ZN_plus}
\end{align}
   By assumption, the Riccati equations~\eqref{eq:riccati},~\eqref{eq:H_riccati} with the terminal conditions $ \Pi_\ft = Y_{\ft,-} $, $ H_\ft = \bsigma_\ft^{-1} - Y_{\ft,-} $ have solutions $ \{\Pi_{k,-} \}_{k=0}^\ft $, $  \{H_{k,-}\}_{k=0}^\ft $. Therefore, by Proposition~\ref{prop:linear}, the solution $ \{\Pi_{k,-}, H_{k,-}\} $ satisfies the boundary conditions~\eqref{eq:boundary}.

   \cyan{Next, we show that the solutions $ X_k = X_k^{\ft,+} $ and $ \what{X}_k = \what{X}_k^{\ft,+} $ to \eqref{eq:XY},~\eqref{eq:XYhat} specified by $ (X_\ft,Y_\ft) = (I,Y_{\ft,+}) $, $ (\what{X}_\ft, \what{Y}_\ft) = (I,\what{Y}_{\ft,+})  $ are invertible for any $ k\in \bbra{0,\ft} $.}
   For $ (X_\ft, Y_\ft) = (I, Y_{\ft,+}) $, we have
   \begin{align}
       X_k^{\ft,+} = \Phi_{11} (k,\ft) + \Phi_{12} (k,\ft) \left( Z_{\ft,+} + \frac{1}{2} \bsigma_\ft^{-1} \right) \label{eq:Xkplus}
   \end{align}
   for any $ k\in \bbra{0,\ft} $.
   By Lemma~\ref{lem:T_invertible}, for any $ k\in \gram{\bbra{0, k_\rmr}} $, $ \Phi_{12} (k,\ft) $ is invertible and
   \begin{align*}
       &\Phi_{12} (k,\ft)^{-1} X_k^{\ft,+} = \Phi_{12} (k,\ft)^{-1} \Phi_{11} (k,\ft) - \varphi_{12}^{-1} \varphi_{11} \\
       & +  \bsigma_\ft^{-1/2} \left( \frac{1}{4}I + \bsigma_\ft^{1/2} \varphi_{12}^{-1} \bsigma_0 \varphi_{12}^{-\top} \bsigma_\ft^{1/2}  \right)^{1/2} \bsigma_\ft^{-1/2} + \frac{1}{2} \bsigma_\ft^{-1}.
   \end{align*}
   In addition, by Lemma~\ref{lem:T_monotone}, for any $ k\in \gram{\bbra{0, k_\rmr}} $, we have
   \begin{align*}
       \Phi_{12} (k,\ft)^{-1} \Phi_{11} (k,\ft) &\succeq \Phi_{12} (0,\ft)^{-1} \Phi_{11} (0,\ft) \\
       &= \varphi_{12}^{-1} \varphi_{11} .
   \end{align*}
   Therefore, it holds that \gram{for any $k\in \bbra{0, k_\rmr}$,}
   \begin{align*}
       &\Phi_{12} (k,\ft)^{-1} X_k^{\ft,+} \\
       &\succeq  \bsigma_\ft^{-1/2} \left( \frac{1}{4}I + \bsigma_\ft^{1/2} \varphi_{12}^{-1} \bsigma_0 \varphi_{12}^{-\top} \bsigma_\ft^{1/2}  \right)^{1/2} \bsigma_\ft^{-1/2} + \frac{1}{2} \bsigma_{\ft}^{-1} \\
       &\succ 0 ,
   \end{align*}
   which means that $ X_k^{\ft,+} $ is invertible for any $ k\in \gram{\bbra{0, k_\rmr}} $.
   By a similar argument, $ \what{X}_k^{\ft,+} $ is shown to be invertible for any $ k\in \gram{\bbra{0, k_\rmr}} $.
   
   We next prove that $ X_k = X_k^{\ft,+} $ is invertible for $ k \in \gram{\bbra{k_\rmr+1 ,\ft}} $.
   \cyan{To this end, we derive another expression of $ X_k^{\ft,+} $ given in \eqref{eq:X0k}, which is different from \eqref{eq:Xkplus}.}
   \rev{By the same argument as for \eqref{eq:terminal1},~\eqref{eq:terminal2}, under $ X_0 = \what{X_0} = I $, the two sets of initial values $ (Y_0,\what{Y}_0) = (Y_{0,\pm}, \what{Y}_{0,\pm}) $ \fin{lead to} the boundary conditions~\eqref{eq:boundary2_1},~\eqref{eq:boundary2_2}}, where \cyan{$ Y_{0,\pm} $ is defined by \eqref{eq:Y_0_pm}} and
   \begin{align}
       \what{Y}_{0,\pm} := Y_{0,\pm} - \bsigma_0^{-1} .
   \end{align}
   In the sequel, we explain that $ (X_\ft, Y_\ft) = (I, Y_{\ft,+}) $ and $ (X_0, Y_0) = (I, Y_{0,+}) $ result in the same solution $ \{\Pi_k\} $.
   The terminal values $ (X_\ft,Y_\ft ) = (I, Y_{\ft,+}) $, $(\what{X}_\ft,\what{Y}_\ft ) = (I, \what{Y}_{\ft,+})  $ lead to the initial values:
   \begin{align*}
       (X_0,Y_0) &= (X_0^{\ft,+}, Y_0^{\ft,+}) \\
       &:= (\varphi_{11} + \varphi_{12} Y_{\ft,+},\varphi_{21} + \varphi_{22} Y_{\ft,+}), \\
       (\what{X}_0, \what{Y}_0) &= (\what{X}_0^{\ft,+}, \what{Y}_0^{\ft,+}) \\
       &:=(\varphi_{11} + \varphi_{12} \what{Y}_{\ft,+},\varphi_{21} + \varphi_{22} \what{Y}_{\ft,+}) .
   \end{align*}
   By the linearity of \eqref{eq:XY},~\eqref{eq:XYhat}, and the invertibility of $ X_0^{\ft,+}, \what{X}_0^{\ft,+} $, \fin{which has already been shown,} it can be easily checked that \cyan{the solutions to \eqref{eq:XY},~\eqref{eq:XYhat} with the initial conditions}
   \begin{align*}
       (X_0, Y_0) &= \left(I, Y_0^{\ft,+}(X_0^{\ft,+})^{-1}\right), \\
       (\what{X}_0, \what{Y}_0) &= \left(I, \what{Y}_0^{\ft,+}(\what{X}_0^{\ft,+})^{-1}\right),
   \end{align*}
       satisfy the boundary conditions~\eqref{eq:boundary2_1},~\eqref{eq:boundary2_2}. \cyan{Since only the two sets $ (Y_0, \what{Y}_0) = (Y_{0,\pm}, \what{Y}_{0,\pm}) $ satisfy the boundary conditions~\eqref{eq:boundary2_1},~\eqref{eq:boundary2_2} under $ X_0 = \what{X}_0 = I $,} one of the following two relationships holds:
   \begin{align}
       &Y_0^{\ft,+}(X_0^{\ft,+})^{-1} = Y_{0,+}, \label{eq:correct}\\
       &Y_0^{\ft,+}(X_0^{\ft,+})^{-1} = Y_{0,-} .
   \end{align}
   In what follows, we show that \eqref{eq:correct} holds, which suggests that $ (X_\ft,Y_\ft) = (I,Y_{\ft,+}) $ and $ (X_0,Y_0) = (I,Y_{0,+}) $ leads to the same initial value $ \Pi_0 = Y_0^{\ft,+}(X_0^{\ft,+})^{-1} = Y_{0,+} $.
   
   Noting that $ \Phi_M (\ft,0) \Phi_M (0,\ft) = I $, we obtain
   \begin{align*}
       \phi_{11} \varphi_{11} + \phi_{12} \varphi_{21} = I,~~\phi_{11} \varphi_{12} + \phi_{12} \varphi_{22} = 0.
   \end{align*}
   Since $ \phi_{12} $ is invertible by Lemma~\ref{lem:T_invertible},
   \begin{align*}
       \varphi_{21} &= \phi_{12}^{-1} ( I - \phi_{11}\varphi_{11}) = \phi_{12}^{-1} - \phi_{12}^{-1} \phi_{11} \varphi_{11} , \\
       \varphi_{22} &= \phi_{12}^{-1} \phi_{12} \varphi_{22} = - \phi_{12}^{-1} \phi_{11} \varphi_{12} .
   \end{align*}
   Then, the left-hand side of \eqref{eq:correct} is rewritten as
   \begin{align}
       &(\phi_{12}^{-1} - \phi_{12}^{-1} \phi_{11} \varphi_{11}- \phi_{12}^{-1} \phi_{11} \varphi_{12} Y_{\ft,+}) (\varphi_{11} + \varphi_{12} Y_{\ft,+})^{-1} \nonumber\\
       &= \phi_{12}^{-1} ( I - \phi_{11} (\varphi_{11} + \varphi_{12} Y_{\ft,+})) (\varphi_{11} + \varphi_{12} Y_{\ft,+})^{-1} \nonumber\\
       &= - \phi_{12}^{-1} \phi_{11} + \phi_{12}^{-1} (\varphi_{11} + \varphi_{12} Y_{\ft,+})^{-1} , \label{eq:left_hand}
   \end{align}
   whose first term coincides with that of $ Y_{0,\pm} $.
   In addition,
   \[
       Y_{0,+} + \phi_{12}^{-1} \phi_{11} \prec 0, ~~ Y_{0,-} + \phi_{12}^{-1} \phi_{11} \succ 0 .
   \]
   Hence, by showing $ \phi_{12}^{-1} (\varphi_{11} + \varphi_{12} Y_{\ft,+})^{-1} \prec 0 $, we get \eqref{eq:correct}.
   In fact, by \cite[Lemma~4]{Liu2022}, \fin{we have} $ \phi_{12} = - \varphi_{12}^\top $, and \fin{thus},
   \begin{align}
       &(\varphi_{11} + \varphi_{12}Y_{\ft,+}) \phi_{12} \nonumber\\
       &= -\varphi_{12} \bsigma_\ft^{-1/2} \left( \frac{1}{4} I + \bsigma_\ft^{1/2} \varphi_{12}^{-1} \bsigma_0 \varphi_{12}^{-\top} \bsigma_\ft^{1/2}  \right)^{1/2} \bsigma_\ft^{-1/2} \varphi_{12}^\top \nonumber\\
       &\quad - \frac{1}{2} \varphi_{12} \bsigma_\ft^{-1} \varphi_{12}^\top \prec 0 . \nonumber 
   \end{align}
   Similarly, we have $ \what{Y}_0^{\ft,+}(\what{X}_0^{\ft,+})^{-1} = \what{Y}_{0,+} $.

   Recall that we would like to show the invertibility of $ X_k = X_k^{\ft,+} $ given by the terminal condition $ (X_\ft,Y_\ft) = (I,Y_{\ft,+}) $, or equivalently, by the initial condition $ (X_0,Y_0) = (X_0^{\ft,+},Y_0^{\ft,+}) $.
   Denote by $ X_k^{0,+} $ the solution to \eqref{eq:XY} specified by $ (X_0,Y_0) = (I,Y_{0,+}) $.
   Then, by \eqref{eq:correct} and the linearity of \eqref{eq:XY}, the initial condition $ (X_0,Y_0) = (X_0^{\ft,+},Y_0^{\ft,+}) $ leads to
   \begin{equation}\label{eq:X0k}
   X_k = X_k^{\ft,+} = X_k^{0,+} X_0^{\ft,+}  , ~~ \forall k \in \bbra{0,\ft} .
   \end{equation}
   Since \cyan{we have shown that} $ X_0^{\ft,+} $ is invertible, if $ X_k^{0,+} $ is invertible, then $ X_k = X_k^{\ft,+} $ given by $ (X_\ft,Y_\ft) = (I, Y_{\ft,+} ) $ is invertible.
   For clarity, the relationships between $ X_k^{\ft,+} $ and $ X_k^{0,+} $ are shown in Fig.~\ref{fig:proof}.
   
   \begin{figure}[!t]
       \centering
       \includegraphics[scale=0.3]{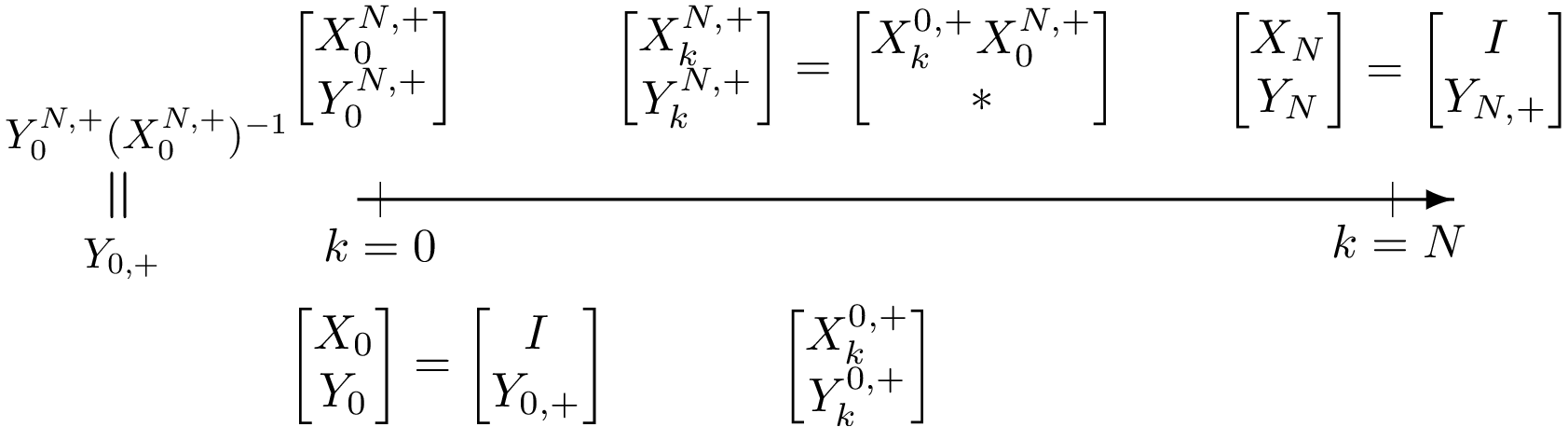}
   \put(-71,60){\eqref{eq:XY}}
   \put(-170,60){\eqref{eq:XY}}
   \put(-145,14){\eqref{eq:XY}}
   \put(-238,13){\eqref{eq:correct}}
   \put(-127,69){\eqref{eq:X0k}}
   \put(-138,9){\vector(1,0){10}}%
   \put(-135,9){\vector(-1,0){10}}%
   \put(-160,55){\vector(-1,0){10}}%
   \put(-160,55){\vector(1,0){10}}%
   \put(-60,55){\vector(-1,0){10}}%
   \put(-63,55){\vector(1,0){10}}%
       \caption{Relationships between the symbols in the proof of Proposition~\ref{thm:riccati_solution}.}
       \label{fig:proof}
   \end{figure}

   By the same argument as for the invertibility of $  X_k  = X_k^{\ft,+} $ for $ k\in \gram{\bbra{0,k_\rmr}} $, we can show that $ X_k = X_k^{0,+} $ is invertible for $ k\in \gram{\bbra
   {k_\rmr+1,\ft}} $. 
   Therefore, by \eqref{eq:X0k}, $ X_k^{\ft,+} $ is invertible for $ k\in \gram{\bbra{k_\rmr+1,\ft}} $.
   Similarly, it can be shown that $ \what{X}_k = \what{X}_k^{\ft,+} $ specified by $ (\what{X}_\ft,\what{Y}_\ft) = (I, \what{Y}_{\ft,+}) $ is invertible for any $ k\in \gram{\bbra{k_\rmr+1, \ft}} $.
   In summary, $ X_k $ and $ \what{X}_k $ with $ (X_\ft,Y_\ft) = (I, Y_{\ft,+}), (\what{X}_\ft,\what{Y}_\ft) = (I, \what{Y}_{\ft,+}) $ are invertible for any $ k\in \bbra{0,\ft}  $.
   This concludes that the Riccati equations \eqref{eq:riccati},~\eqref{eq:H_riccati} coupled through \eqref{eq:boundary} have a solution given by the terminal values~\eqref{eq:Pi_terminal},~\eqref{eq:H_terminal}.
   Finally, \eqref{eq:correct} means $ \Pi_{0,+} = Y_{0,+} $.

   \section{Proof of Proposition~\ref{prop:positive}}\label{app:positive}
   Noting that $ \Pi_{0,+} \prec -\phi_{12}^{-1} \phi_{11}$, $\Pi_{0,-} \succ -\phi_{12}^{-1} \phi_{11} $, we can obtain Proposition~\ref{prop:positive} by \cite[Corollary~3]{Liu2022}. Nevertheless, it would be insightful to give another proof utilizing the equivalent backward problem (Problem~\ref{prob:inverse}) and the explicit form of the solution $ \{\Pi_k\} $. \\
       {\bf (Proof of (i))}
       We prove \eqref{eq:BPiB_positive} by induction on $ k $. Before that, we show some properties of $ T_{k} := \Phi_{11} (0,k)^{-1} \Phi_{12} (0,k) $.
       By Lemmas~\ref{lemma:phi11_inv},~\ref{lem:T_invertible}, it holds $ T_k \succ 0 $ for $ k \in \bbra{k_\rmr,\ft} $. In addition, by \eqref{eq:riccati_T2}, we have
       \begin{equation}\label{eq:TG_positive}
       T_{k+1} = G_k + \bara_k (T_k^{-1} + \barq_k)^{-1} \bara_k^\top , \ \forall k \in \fin{\bbra{k_\rmr,\ft-1}},
       \end{equation}
   and therefore, $ T_{k+1} - G_k \succ 0 $, \fin{$ k\in \bbra{k_\rmr,\ft-1} $}. This implies
   \begin{align}
       &( I - B_{k-1}^\top T_k^{-1} B_{k-1})^{-1} \nonumber\\
       &= I - B_{k-1}^\top (- T_k + G_{k-1})^{-1} B_{k-1} \succ 0 , ~~ \forall k \in \fin{\bbra{k_\rmr+1,\ft}} . \label{eq:positive}
   \end{align}
   In addition, by \eqref{eq:TG_positive}, it holds that for any $ k\in \fin{\bbra{k_\rmr,\ft-1}} $,
   \begin{align*}
       &\bara_k^{-\top} (T_k^{-1} + \barq_k) \bara_k^{-1} \\
       &= T_{k+1}^{-1} - T_{k+1}^{-1} B_k (- I + B_k^\top T_{k+1}^{-1} B_k)^{-1} B_k^\top T_{k+1}^{-1} ,
   \end{align*}
   which yields
   \begin{align}
       \esu_k &= \bara_k^\top \esu_{k+1} \bara_k \nonumber\\
       &\quad - \bara_k^\top \esu_{k+1} B_k (I + B_k^\top \esu_{k+1} B_k)^{-1} B_k^\top \esu_{k+1} \bara_k + \barq_k \nonumber\\
       &=: \rev{\scrR_k} (\esu_{k+1}) , ~~  \forall k\in \fin{\bbra{k_\rmr,\ft-1}} , \label{eq:riccati_S}
   \end{align}
   where we defined $ \esu_k := - T_k^{-1} $, \fin{$ k \in \bbra{k_\rmr,\ft} $}. 
   Recall that $ \Pi_k = \scrR_k (\Pi_{k+1}) $. 
   Let $ \Delta_k := \Pi_k - \esu_k $. By \cite[Lemma~2.2]{Wimmer1992}, we have \fin{for any $\bbra{k_\rmr,\ft-1}$},
   \begin{align}
       &\Pi_k - \esu_k = \scrR_{k} (\Pi_{k+1}) - \scrR_{k} (\esu_{k+1}) \nonumber\\
       &= \calA_k (\Pi_{k+1})^\top \Delta_{k+1} \calA_k (\Pi_{k+1}) + \calA_k (\Pi_{k+1})^\top \Delta_{k+1} B_k \nonumber\\
       &\times ( I + B_k^\top \esu_{k+1} B_{k})^{-1}  B_k^\top \Delta_{k+1} \calA_k (\Pi_{k+1})  . \label{eq:comparison}
   \end{align}

   Now, we are ready to perform induction.
   First, for $ \Pi_{\ft,+} $ and $ \esu_\ft $, we have
       \begin{align}
           \Pi_{\ft,+} &= \esu_\ft + \frac{1}{2} \bsigma_\ft^{-1} \nonumber\\
           &\quad+  \bsigma_\ft^{-1/2} \left( \frac{1}{4}I + \bsigma_\ft^{1/2} \varphi_{12}^{-1} \bsigma_0 \varphi_{12}^{-\top} \bsigma_\ft^{1/2}  \right)^{1/2} \bsigma_\ft^{-1/2} \nonumber\\
           &\succ \esu_\ft . \label{eq:PiN_SN}
       \end{align}
   Next, we assume the induction hypothesis that for some $ k = \bar{k} \in \bbra{k_\rmr+1, \ft} $, it holds that
   \begin{equation}\label{eq:Pi_S_k}
       \Pi_{\bar{k},+} \succ \esu_{\bar{k}} .
   \end{equation}
   Then, by \eqref{eq:positive}, we have
   \begin{align}
       I +B_{\bar{k}-1}^\top \Pi_{\bar{k},+} B_{\bar{k}-1} \succeq I  +B_{\bar{k}-1}^\top \esu_{\bar{k}} B_{\bar{k}-1} \succ 0 .  \label{eq:BPiB_k_positive}
   \end{align}
   By \eqref{eq:comparison},~\eqref{eq:Pi_S_k},~\eqref{eq:BPiB_k_positive}, and the invertibility of $ \calA_{\bar{k}-1} (\Pi_{\bar{k},+}) $, it holds $ \Pi_{\bar{k}-1,+} \succ \esu_{\bar{k}-1} $. 
   This completes the induction step, and we obtain \eqref{eq:BPiB_positive} for \fin{$ k\in \bbra{k_\rmr,\ft-1} $}.

   Next, to prove \eqref{eq:BPiB_positive} for \fin{$ k \in \bbra{0,k_\rmr-1} $}, we utilize the equivalent backward problem~(Problem~\ref{prob:inverse}).
   \fin{By Proposition~\ref{prop:opt_inv}, it suffices to show that $ P_k := \Pi_{k,+} - \barq_k $ satisfies $ I - B_k^\top \bara_k^{-\top} P_k \bara_k^{-1} B_k \succ 0 $, which is equivalent to \eqref{eq:BPiB_positive}.
   In the following, we derive another expression of the initial value $ P_0 $ rather than $ \Pi_{0,+} - \barq_0 $ for proving $ I - B_k^\top \bara_k^{-\top} P_k \bara_k^{-1} B_k \succ 0 $.}
   Let $ \wtilde{\phi}_{ij} := \wtilde{\Phi}_{ij} (\ft,0), $ $(i,j=1,2) $, where
   \begin{align}
       \wtilde{\Phi} (k,l) &:=
               \begin{cases}
                   \wtilde{M}_{k-1}^{-1} \wtilde{M}_{k-2}^{-1} \cdots \wtilde{M}_{l}^{-1}, & k > l,\\
                   I, & k=l , \\
                   \wtilde{M}_{k} \wtilde{M}_{k+1} \cdots \wtilde{M}_{l-1}, & k < l ,
               \end{cases} \label{eq:rev_transition}\\
               \wtilde{M}_k &:=
       \begin{bmatrix}
           \bara_k^{-1} + \bara_k^{-1}B_kB_k^\top \barq_{k+1} & ~~ -\bara_k^{-1}B_kB_k^\top \\
           - \bara_k^{\top} \barq_{k+1} & \bara_k^{\top}
       \end{bmatrix} .\nonumber
   \end{align}
   The reachability Gramian for the backward system~\eqref{eq:system_inv} is given by
   \begin{equation*}
       \wtilde{\calR}(k_1,k_0) := \sum_{k = k_0}^{k_1-1} \wtilde{\Phi}_{\tilde{A}}(k_0,k) \wtilde{B}_k\wtilde{B}_k^\top  \wtilde{\Phi}_{\tilde{A}}(k_0,k)^\top , \  k_1 > k_0,
   \end{equation*}
   where $\wtilde{B}_k := -\bara_k^{-1} B_k $ and $ \wtilde{\Phi}_{\tilde{A}} (k_0,k) $ denotes the time-reversed transition matrix for $ \{\wtilde{A}_k\} $, \fin{$ \wtilde{A}_k := \bara_k^{-1} $} defined like in \eqref{eq:rev_transition}. Then, it can be verified that
   \begin{equation}\label{eq:controllability}
       \wtilde{\calR}(k_1,k_0) = \Phi_{\bara} (k_0,k_1) \calR(k_1,k_0) \Phi_{\bara} (k_0,k_1)^\top .
   \end{equation}
   Therefore, under Assumption~\ref{ass:reachability}, $ \wtilde{\calR}(k,0) $ is invertible for any $ k\in \bbra{k_\rmr,\ft} $, and $ \wtilde{\calR} (\ft,k) $ is invertible for any \gram{$ k\in \bbra{0,k_\rmr} $}.
   By applying Lemma~\ref{lem:T_invertible} to the backward problem, we deduce that $ \wtilde{T}_k := \wtilde{\Phi}_{11} (\ft,k)^{-1} \wtilde{\Phi}_{12} (\ft,k) $ is invertible for \gram{$ k\in \bbra{0,k_\rmr} $}, and thus, $\wtilde{\phi}_{12} $ is also invertible.
   Moreover, similar to \eqref{eq:riccati_S}, $ \wtilde{\esu}_k := - \wtilde{T}_k^{-1}, \ \gram{k\in \bbra{0,k_\rmr}} $ satisfies
   \begin{align}
       &\wtilde{\esu}_{k+1} = \bara_k^{-\top} \wtilde{\esu}_k \bara_k^{-1} - \bara_k^{-\top} \wtilde{\esu}_k \bara_k^{-1} B_k \bigl(I + B_k^\top \bara_k^{-\top} \wtilde{\esu}_k  \nonumber\\
       &\times \bara_k^{-1} B_k\bigr)^{-1} B_k^\top \bara_k^{-\top} \wtilde{\esu}_k \bara_k^{-1} + \barq_{k+1} ,  ~~ \forall k\in \gram{\bbra{0,k_\rmr-1}},   \label{eq:S_riccati}\\
       &\wtilde{\esu}_0 = - \wtilde{\phi}_{12}^{-1} \wtilde{\phi}_{11} . \nonumber
   \end{align}

   By the same argument as for the forward problem (Problem~\ref{prob:density_control}), two candidates for the initial condition \fin{$ (J_0,P_0) = (J_{0,\pm}, P_{0,\pm})  $} under which the solution $ \{J_k,P_k\} $ to the Riccati equations~\eqref{eq:riccati_inv},~\eqref{eq:riccati_inv_dual} satisfies the boundary conditions~\eqref{eq:boundary_inv}, are given by
   \begin{align*}
       J_{0,\pm} &:= -\wtilde{\phi}_{12}^{-1} \wtilde{\phi}_{11} + \frac{1}{2} \bsigma_0^{-1}\\
       &\quad\pm \bsigma_0^{-1/2} \left( \frac{1}{4} I + \bsigma_0^{1/2} \wtilde{\phi}_{12}^{-1} \bsigma_\ft \wtilde{\phi}_{12}^{-\top} \bsigma_0^{1/2} \right)^{1/2} \bsigma_0^{-1/2}  ,\\
       P_{0,\pm} &:= \bsigma_0^{-1} - J_{0,\pm} \\
       &= \wtilde{\phi}_{12}^{-1} \wtilde{\phi}_{11} +  \frac{1}{2} \bsigma_0^{-1} \\
       &\quad\mp \bsigma_0^{-1/2} \left( \frac{1}{4} I + \bsigma_0^{1/2} \wtilde{\phi}_{12}^{-1} \bsigma_\ft \wtilde{\phi}_{12}^{-\top} \bsigma_0^{1/2} \right)^{1/2} \bsigma_0^{-1/2}   . 
   \end{align*}

   On the other hand, by Propositions~\ref{thm:riccati_solution},~\ref{prop:opt_inv}, \fin{the initial condition $ \wtilde{P}_{0,+} $ of \eqref{eq:riccati_inv_dual} where}
   \begin{align*}
       \wtilde{P}_{0,+} &:= \Pi_{0,+} - \barq_0 \\
       &= - \phi_{12}^{-1} \phi_{11} + \frac{1}{2} \bsigma_0^{-1} - \barq_0\\
       &\quad - \bsigma_0^{-1/2} \left( \frac{1}{4}I + \bsigma_0^{1/2} \phi_{12}^{-1} \bsigma_\ft \phi_{12}^{-\top} \bsigma_0^{1/2}  \right)^{1/2} \bsigma_0^{-1/2} ,
   \end{align*}
   \fin{also leads to the boundary conditions~\eqref{eq:boundary_inv}, and thus, satisfies} either $ P_{0,+} = \wtilde{P}_{0,+} $ or $ P_{0,-} = \wtilde{P}_{0,+} $. 
   \fin{Let us regard $ P_{0,\pm} $ and $ \wtilde{P}_{0,+} $ as functions of $ \bsigma_\ft $. Assume that $ P_{0,-} = \wtilde{P}_{0,+} $ for some $ \bsigma_\ft $. Then, by the continuity of $ P_{0,\pm} $, $ \wtilde{P}_{0,+} $ in $ \bsigma_\ft $ and $ P_{0,-} \succ P_{0,+} $, it holds that $ P_{0,-} = \wtilde{P}_{0,+} $ for any $ \bsigma_\ft $. However, for sufficiently large $ \bsigma_\ft $, $ \wtilde{P}_{0,+} $ is negative definite while $ P_{0,-} $ is positive definite.}
   \fin{This is a contradiction, and thus, it holds that}
   \begin{equation}\label{eq:P0_Pi0}
       P_{0,+} = \wtilde{P}_{0,+} = \Pi_{0,+} - \barq_0 .
   \end{equation}
   \fin{Therefore, the solution $ P_k = \Pi_{k,+} - \barq_k $ to \eqref{eq:riccati_inv_dual} satisfies $ P_0 = P_{0,+} $.}

   For such $ \{P_k\} $, let $ U_k := - P_k$, $U_0 = - P_{0,+} $. Then we have
   \begin{align}
       &U_{k+1} = \bara_k^{-\top} U_k \bara_k^{-1} - \bara_k^{-\top} U_k \bara_k^{-1} B_k \nonumber\\
       &\times (I + B_k^\top \bara_k^{-\top} U_k \bara_k^{-1} B_k)^{-1} B_k^\top \bara_k^{-\top} U_k \bara_k^{-1} + \barq_{k+1}  \label{eq:U_riccati}
   \end{align}
   for any $ k\in \bbra{0,\ft-1} $.
   Noting that $ U_0 \succ \wtilde{\esu}_0 $, by \eqref{eq:S_riccati},~\eqref{eq:U_riccati}, and the same argument as for $ \Pi_{k,+}, \esu_k $, we obtain
   \begin{align}
       I + B_k^\top \bara_k^{-\top} U_k \bara_k^{-1} B_k &\succeq I + B_k^\top \bara_k^{-\top} \fin{\wtilde{\esu}_k} \bara_k^{-1} B_k \succ 0 , \nonumber\\
       &\qquad \forall k\in \bbra{0,k_\rmr -1}, \label{eq:BUB_positive}
   \end{align}
   which implies $ I + B_k^\top \Pi_{k+1,+} B_k \succ 0 $ for any $ k\in \bbra{0,k_\rmr -1} $ \fin{by Proposition~\ref{prop:opt_inv}.}
   The proof of \eqref{eq:BHB_positive} is the same as that of \eqref{eq:BPiB_positive} by exchanging the roles of Problems~\ref{prob:density_control},~\ref{prob:inverse}. \cyan{In fact, by the same argument as for \eqref{eq:BPiB_positive}, it can be shown that for the solution $ \{J_k\} $ to \eqref{eq:riccati_inv} satisfying $ J_k = H_{k,+} + \barq_k $, it holds $ I + B_k^\top \bara_k^{-\top} J_k \bara_k^{-1} B_k \succ 0 $ for any $ k\in \bbra{0,\ft-1} $. \fin{This means \eqref{eq:BHB_positive} by Proposition~\ref{prop:opt_inv}, which completes the proof of (i).}}

   {\bf (Proof of (ii))}
   We prove the statement (ii) by contradiction. Assume that for any $ k\in \bbra{0,\ft-1} $, it holds $ I + B_{k}^\top \Pi_{k+1,-} B_{k} \succ 0  $.
       By $ \Pi_{\ft,-} \prec \esu_\ft $, $ I + B_{\ft-1}^\top \Pi_{\ft,-} B_{\ft-1} \succ 0  $, and
       \begin{align}
           &\esu_k - \Pi_k = -\calA_k (\esu_{k+1})^\top \Delta_{k+1} \calA_k (\esu_{k+1}) \nonumber\\
           &+ \calA_k (\esu_{k+1})^\top \Delta_{k+1} B_k ( I + B_k^\top \Pi_{k+1} B_{k})^{-1}  B_k^\top \nonumber\\
           &\times \Delta_{k+1} \calA_k (\esu_{k+1}) ,~~ \forall k\in \fin{\bbra{k_\rmr,\ft-1}} , \label{eq:comparison2} 
       \end{align}
       where $ \Delta_k = \Pi_k - \esu_k $, we obtain $ \Pi_{\ft-1,-} \prec \esu_{\ft-1} $. By induction and $ \esu_{k} = - T_k^{-1} \prec 0 $, $ k \in \bbra{k_\rmr,\ft} $, it holds that
       \begin{equation}\label{eq:Pi_negative}
           \Pi_{k,-} \prec \esu_k \prec 0  , ~~ \forall k\in \bbra{k_\rmr,\ft} .
       \end{equation}
   
       Let $ \{U_k\} =  \{U_{k,-}\} $ be the solution to \eqref{eq:U_riccati} with the initial condition $ U_0 = -P_{0,-} $. By \eqref{eq:P0_Pi0}, it \fin{must} hold $ P_{0,-} = \Pi_{0,-} - \barq_0 $, and hence by Proposition~\ref{prop:opt_inv}, $ U_{k,-} = -(\Pi_{k,-} - \barq_k) $ for any $ k\in \bbra{0,\ft} $.
       Note that $ U_{0,-} \prec \wtilde{\esu}_0 $, and $ I + B_k^\top \Pi_{k+1,-} B_k \succ 0 $ is equivalent to $ I + B_k^\top \bara_k^{-\top} U_{k,-} \bara_k^{-1} B_k \succ 0  $. Then, by \eqref{eq:S_riccati},~\eqref{eq:U_riccati}, and the same argument as for $ \Pi_{k,-} $ and $ \esu_k $, we get
       \begin{equation}
           U_{k,-} \prec \wtilde{\esu}_k \prec 0 , ~~ \forall k \in \gram{\bbra{0,k_\rmr}} .
       \end{equation}
       Therefore, we have $\Pi_{k_\rmr,-} \succ \barq_{k_\rmr} \succeq 0 $, which contradicts \eqref{eq:Pi_negative}. This completes the proof.

   \section{Proof of Proposition~\ref{prop:opt_inv} and \eqref{eq:opt_forward_another}}\label{app:reverse}
   In this appendix, when no confusion is possible, we will drop the subscript $ k $.
   The Riccati equation~\eqref{eq:H_riccati} can be written as
   \begin{align*}
       H_k =\modi{\bara^\top} H_{k+1} \bara(I - \bara^{-1} BB^\top H_{k+1} \bara)^{-1} - \barq ,
   \end{align*}
   which yields
   \begin{align*}
       (H_k + \barq_k) (I - \bara^{-1} BB^\top H_{k+1}\bara) = \bara^\top H_{k+1} \bara .
   \end{align*}
   Therefore, we have
   \begin{equation}\label{eq:HQ}
       H_k + \barq_k = (\bara^\top + (H_k + \barq_k) \bara^{-1}BB^\top) H_{k+1} \bara .
   \end{equation}
   Here, $ C_k := \bara^\top + (H_k + \barq_k) \bara^{-1}BB^\top $ is invertible. In fact, formally it holds that
   \begin{align}
       C_k^{-1} &= \bigl[I - \bara^{-\top} (H_k + \barq_k) \bara^{-1} B \nonumber\\
       &\times (I + B^\top \bara^{-\top}(H_k  +\barq_k)\bara^{-1} B)^{-1} B^\top\bigr] \bara^{-\top} .\label{eq:Cinv}
   \end{align}
   Hence, the invertibility of $ I + B^\top \bara^{-\top}(H_k  +\barq_k)\bara^{-1} B $ implies that $ C_k $ is invertible.
   Moreover, we have
   \begin{align*}
       &I + B^\top \bara^{-\top}(H_k  +\barq_k)\bara^{-1} B \\
       &= I + B^\top H_{k+1} (I - BB^\top H_{k+1})^{-1} B ,
   \end{align*}
   whose right-hand side has the inverse matrix $ I - B^\top H_{k+1} B $.
   Therefore, $ C_k $ is invertible, and by \eqref{eq:HQ} and \eqref{eq:Cinv},
   \begin{align}
       &H_{k+1} = \bigl[I - \bara_k^{-\top} (H_k + \barq_k)\bara_k^{-1} B_k \nonumber\\
       &\times (I + B_k^\top \bara_k^{-\top} (H_k + \barq_k) \bara_k^{-1} B_k)^{-1}B_k^\top\bigr] \bara_k^{-\top} \nonumber\\
       &\times (H_k + \barq_k) \bara_k^{-1} , \label{eq:H}
   \end{align}
   which means $ J_k := H_k + \barq_k $ is a solution to \eqref{eq:riccati_inv}.
   Similar to \eqref{eq:H}, we obtain
   \begin{align*}
       &\Pi_{k+1} = \bigl[I - \bara_k^{-\top} (\Pi_k - \barq_k) \bara_k^{-1} B_k \\
       &\times(-I + B_k^\top \bara_k^{-\top} (\Pi_k - \barq_k) \bara_k^{-1} B_k)^{-1} B_k^\top \bigr]\bara_k^{-\top} \\
       &\times (\Pi_k - \barq_k) \bara_k^{-1}  .
   \end{align*}
   Hence, $ P_k :=  \Pi_k - \barq_k $ satisfies \eqref{eq:riccati_inv_dual}.
   Moreover, \rev{a straightforward calculation with \eqref{eq:riccati_inv_dual} yields that}
   \begin{align}
       &I + B_k^\top \Pi_{k+1} B_k = I + B_k^\top (P_{k+1} + \barq_{k+1}) B_k \nonumber\\
       &= (I - B_k^\top \bara_k^{-\top} P_k \bara_k^{-1} B_k)^{-1} . \label{eq:Pi_P_positive}
   \end{align}
   Similarly, for $ J_k = H_k + \barq_k $, we have
   \begin{equation}\label{eq:H_J}
       I - B_k^\top H_{k+1} B_k = (I + B_k^\top \bara_k^{-\top} J_k \bara_k^{-1} B_k)^{-1} ,
   \end{equation}
   and thus, $ I - B_k^\top H_{k+1} B_k \succ 0 $ is equivalent to $ I + B_k^\top \bara_k^{-\top} J_k \bara_k^{-1} B_k \succ 0 $.

   Furthermore, by the boundary conditions~\eqref{eq:boundary} for $ \{\Pi_k\}, \{H_k\} $, $ J_k = H_k + \barq_k $ and $ P_k = \Pi_k - \barq_k $ satisfy the boundary conditions~\eqref{eq:boundary_inv}.

   By \eqref{eq:H_riccati} and \eqref{eq:H_J}, the mean of the policy~\eqref{eq:opt_policy_reverse} can be written as
   \begin{align}
       &\left(I - B_k^\top H_{k+1} B_k\right) B_k^\top \bara_k^{-\top} (H_k + \barq_k) \bara_k^{-1} \xi_{k+1} \nonumber\\
       &= B_k^\top H_{k+1} \xi_{k+1} . \label{eq:opt_rev_rewrite}
   \end{align}
   Hence, by Corollary~\ref{cor:reverse}, \eqref{eq:opt_policy_inv} is the unique optimal policy of Problem~\ref{prob:inverse}.
   
   Lastly, for the optimal backward state process $ \{\xi_k^*\} $ driven by \eqref{eq:opt_policy_inv} and the optimal forward state process $ \{x_k^*\} $ given in Proposition~\ref{prop:opt}, we have
   \begin{equation}
       \bbE[\xi_k^* (\xi_k^*)^\top] = (J_k + P_k)^{-1} = (H_k + \Pi_k)^{-1} = \bbE[x_k^* (x_k^*)^\top],
   \end{equation}
   which completes the proof of Proposition~\ref{prop:opt_inv}.

   \revv{By the same argument as for \eqref{eq:opt_rev_rewrite} together with \eqref{eq:riccati_inv_dual} and \eqref{eq:Pi_P_positive}, we obtain \eqref{eq:opt_forward_another}.}

   %%%%%%%%%%%%%%%%%%%%%%%%%%%%%%%%%%%%%%%%%%%%%%%%%%%%%%%%%%%%%%%%%%%%%%%%%%%%%%%%

   %%%%%%%%%%%%%%%%%%%%%%%%%%%%%%%%%%%%%%%%%%%%%%%%%%%%%%%%%%%%%%%%%%%%%%%%%%%%%%%%
   %\section*{References}

\section*{References}
\vspace{-\baselineskip}
\bibliographystyle{IEEEtran}
\bibliography{TAC_maxent_state_cost}

\begin{IEEEbiography}{Kaito Ito} (Member, IEEE) received the bachelor's degree in engineering, and the master's and doctoral degrees in informatics from Kyoto University, Kyoto, Japan, in 2017, 2019, and 2022, respectively.
    He was an Assistant Professor at the Department of Computer Science, Tokyo Institute of Technology, Yokohama, Japan in 2022--2024.
	He is currently an Assistant Professor at the Department of Information Physics and Computing, The University of Tokyo, Tokyo, Japan. His research interests include stochastic control, optimization, and machine learning.
\end{IEEEbiography}

\begin{IEEEbiography}{Kenji Kashima} (Senior Member, IEEE) received his Doctoral degree in Informatics from Kyoto University in 2005. He was with Tokyo Institute of Technology and Osaka University, before he joined Kyoto University in 2013, where he is currently an Associate Professor. He was an Alexander von Humboldt Research Fellow at Universit\"at Stuttgart, Germany. His research interests include control and learning theory for complex dynamical systems, and their applications. He received IEEE Control Systems Society (CSS) Roberto Tempo Best CDC Paper Award, Pioneer Award of SICE Control Division, and so on. He has served as an Associate Editor of IEEE Transactions on Automatic Control, IEEE Control Systems Letters, and IEEE CSS Conference Editorial Board. He is a member of the steering committee of MTNS.
\end{IEEEbiography}

% \begin{IEEEbiographynophoto}{Second B. Author,} photograph and biography not available at the
% time of publication.
% \end{IEEEbiographynophoto}

% \begin{IEEEbiographynophoto}{Third C. Author Jr.} (Senior Member, IEEE), photograph and biography not available at the
% time of publication.
% \end{IEEEbiographynophoto}

\end{document}